\documentclass[preprint,12pt]{elsarticle}

\usepackage{amssymb}
\usepackage{amsmath}

\usepackage{amsfonts,amssymb,amsmath,mathrsfs,graphicx}

\usepackage{titlesec}
\usepackage{amsthm}
\usepackage{amsmath}
\usepackage{amssymb}
\usepackage{color}
\usepackage{mathtools}
\usepackage{geometry}

\numberwithin{equation}{section}

\newtheorem{theorem}{Theorem}[section]

\newtheorem{lemma}[theorem]{Lemma}
\newtheorem{proposition}[theorem]{Proposition}

\newtheorem{definition}[theorem]{Definition}
\newtheorem{remark}[theorem]{Remark}

\newcommand{\bfx}{\boldsymbol{x}}
\newcommand{\bfz}{\boldsymbol{z}}
\newcommand{\bfv}{\boldsymbol{v}}
\newcommand{\bfn}{\boldsymbol{n}}

\newcommand{\bfu}{\boldsymbol{u}_o}
\newcommand{\bfuu}{\boldsymbol{u}}
\newcommand{\bff}{\boldsymbol{f}}
\newcommand{\bfF}{\boldsymbol{F}}
\newcommand{\bfM}{\boldsymbol{M}}

\newcommand{\bbr}{\mathbb R}

\usepackage{array}
\newcolumntype{C}[1]{>{\centering\arraybackslash}m{#1}}
\usepackage{graphics}
\usepackage{amssymb, latexsym, amsmath, amsfonts, epsfig}
\usepackage{bm}
\usepackage{graphicx}
\usepackage{color}
\usepackage{epstopdf}
\usepackage{comment}
\usepackage{soul}
\usepackage[normalem]{ulem}
\usepackage{float}

\newcommand{\rev}[1]{{\color{black}#1}}

\def\d{{\, \rm d}}

\newcommand{\M}{\mathcal{M}}
\newcommand{\E}{\mathcal{E}}

\journal{Physica D: Nonlinear Phenomena}

\begin{document}

\begin{frontmatter}


 \title{Particle, kinetic and hydrodynamic models for sea ice floes. Part I: \rev{N}on-rotating floes}
\author{Quanling Deng\fnref{qd}}
  \fntext[qd]{School of Computing, Australian National University, Canberra, ACT 2601, Australia. \rev{Present address: Yau Mathematical Sciences Center, Tsinghua University, Beijing, 100084, China.} }
 \ead{quanling.deng@anu.edu.au; qdeng12@gmail.com}
  
\author{Seung-Yeal Ha\fnref{sh}}
\fntext[sh]{Department of Mathematical Sciences and Research Institute of Mathematics, Seoul National University, Seoul 08826, Republic of Korea.}
 \ead{syha@snu.ac.kr}




%
%
%

\begin{abstract}
We introduce a comprehensive modeling framework for the dynamics of sea ice floes using particle, kinetic, and hydrodynamic approaches. Building upon the foundational work of Ha and Tadmor on the Cucker-Smale model for flocking, we derive a Vlasov-type kinetic formulation and a corresponding hydrodynamic description. The particle model incorporates essential physical properties of sea ice floes, including size, position, velocity, and interactions governed by Newtonian mechanics. By extending these principles, the kinetic model captures large-scale features through the phase-space distribution, and we also present a hydrodynamic model using the velocity moments and a suitable closure condition. In this paper, as an idea-introductory step, we assume that ice floes are non-rotating and focus on the linear velocity dynamics. Our approach highlights the role of contact forces, ocean drag effects, and conservation laws in the multiscale description of sea ice dynamics, offering a \rev{potential pathway} for the improved understanding and prediction of sea ice behaviors in changing climatic conditions.
\end{abstract}




\end{frontmatter}



\section{Introduction} 
The dynamics of sea ice play a pivotal role in Earth's climate system, influencing ocean circulation, heat exchange, and ecological processes in polar regions. The behavior of sea ice is complex, shaped by interactions between individual ice sheets/floes and environmental forces such as atmospheric wind, ocean currents, and temperature gradients. To better understand and predict these dynamics, a variety of modeling approaches have been developed, ranging from continuum hydrodynamic models \cite{dansereau2016maxwell,hibler1979dynamic,hunke1997elastic} to particle (discrete element) models \cite{damsgaard2018application, herman2016discrete, hopkins2004discrete, lindsay2004new, manucharyan2022subzero}.

The accurate simulation of sea ice dynamics is critical for understanding polar climate systems, particularly in the Marginal Ice Zone (MIZ), where sea ice transitions to open ocean. In such regions, the ice cover is highly fragmented and best described as a collection of interacting floes rather than as a continuous medium. Traditional continuum-based models \rev{\cite{nye1959motion, hibler1979dynamic, nye1991topology, feltham2008sea, hunke2015cice}}, while successful in capturing large-scale behaviors of compact ice packs, rely on the continuum assumption.
However, continuum assumption may break down at both small scale, where individual floe interactions are dominant, and large scale, where floe heterogeneity and fracture patterns become significant. 
In contrast, particle-based models \cite{hopkins2004discrete,wilchinsky2006modelling,herman2016discrete,de2024modelling} represent sea ice as a discrete entity and offer a more physically consistent framework for capturing the granular nature of floe dynamics, including collisions, fracturing, and ridging. These models are particularly well-suited for studying MIZ processes, where the mechanical and thermodynamic behaviors of individual floes play a pivotal role in determining the evolution of the ice cover. Developing and refining particle-based simulations is thus essential for bridging different scales and improving predictive capabilities in sea ice modeling \cite{blockley2020future,golden2020modeling,hunke2010sea}.

Recent studies have highlighted the importance of multiscale approaches in capturing the essential features of sea ice dynamics. Particle models, such as those inspired by the Cucker-Smale flocking dynamics, offer an analytical framework for understanding interactions at the particle level while enabling the derivation of continuum models via the corresponding kinetic equation~\cite{ha2008Kinetic}. These models incorporate fundamental physical properties such as mass, position, and velocity, and account for forces including inter-floe contact forces and environmental drag. Building on this foundation, kinetic formulations provide a bridge to hydrodynamic descriptions from particle descriptions, allowing for the integration of large-scale features into a cohesive particle-continuum multiscale framework \cite{deng2024particle}.

Kinetic theory provides a foundational framework for bridging small-scale particle dynamics and large-scale continuum descriptions in complex systems. 
Originating from the statistical mechanics of gases, kinetic equation models the evolution of \rev{one-particle distribution function} in phase space, and it has been successful in the derivation of fluid equations such as the Navier–Stokes and Euler equations \cite{cercignani2013mathematical, golse2005boltzmann, golse2016dynamics, saint2009hydrodynamic}. 
A central analytical tool in this transition is the BBGKY hierarchy (Bogoliubov–Born–Green–Kirkwood–Yvon, \cite{Bogoliubov1946,bogoliubov1946kinetic,born1946general,kirkwood1946statistical,yvon1935theorie}), which systematically connects the Liouville equation to limiting marginal distribution functions. By incorporating appropriate closure assumptions and scaling limits, such as the Boltzmann–Grad limit, one can obtain kinetic equations like the Boltzmann equation, from which hydrodynamic limits can be derived via the Chapman–Enskog and Hilbert expansions, moment methods, or entropy-based approaches. 
In recent decades, kinetic theory has been extended beyond classical gases to a wide range of interacting particle systems, including flocking models \cite{ ha2009simple, ha2008Kinetic, haskovec2013flocking, park2010cucker} and granular media \cite{dufty2000statistical, dufty2001kinetic, hinrichsen2006physics}.

In sea ice floe dynamics, kinetic theory provides a powerful framework for bridging the gap between discrete floe-level interactions and continuum-scale descriptions of sea ice dynamics and rheology. In this work, we initiate a rigorous systematic investigation into the application of kinetic theory to sea ice modeling by introducing a particle-to-continuum pathway that captures the multiscale nature of floe dynamics. 
As a foundational step, we consider a simplified \rev{set-up} involving non-rotating and colliding cylindrical floes. 
Cylindrical floes are widely used in discrete element modeling (DEM; see, for example, \cite{damsgaard2018application,herman2016discrete}). 
This idealized model allows us to focus on the essential features of floe interactions while establishing a rigorous mathematical framework for the development of a continuum hydrodynamic model. 
%
We consider ocean drag forces and assume one-way coupling that the ocean drags the floes but the floe does not impact the ocean. We ignore atmospheric drag forces for simplicity. They can be included, and the derivations remain in a similar fashion to the case of ocean drag.
We adopt \rev{distribution laws} for initializing the floes, namely power-law distributions \cite{alberello2022three,roach2018emergent,meylan2021floe} for floe sizes and gamma distributions \cite{bourke1987sea, toyota2011size} for floe thickness. 
\rev{Evolution models for the floe size distribution \cite{perovich2014seasonal,horvat2017evolution} and for the floe thickness distribution \cite{toppaladoddi2015theory,toppaladoddi2023seasonal} can be incorporated into the particle framework to account for processes such as melting and refreezing.}
Using a particle model \cite{chen2022superfloe, chen2021lagrangian, damsgaard2018application, kruggel2007review} as the foundation and kinetic theory as a tool, we derive a Vlasov-type kinetic model and the corresponding hydrodynamic model.  
With the particle model and an assumption that the particle number is sufficiently large, we derive the kinetic description based on the distribution function in phase space consisting of the radius, thickness, displacement, and velocity of sea ice floes. Then, the equations of lower-order moments are derived by integrating phase space, which captures the large-scale features of the particles and serves as the coarse-scale continuum model.
This approach highlights the influence of contact forces, drag effects, and conservation laws on the multiscale dynamics of sea ice, paving the way for more realistic and predictive models in future.

The rest of the paper is organized as follows. In Section \ref{sec:pmodel}, we introduce the particle model for colliding and non-rotating sea ice floes in an idealized setting. It is often referred to as the discrete element model (see Sec. \ref{sec:dem}). Then we establish several results for the asymptotic behavior of the dynamics. In Section \ref{sec:kmodel} and Section \ref{sec:hmodel}, under the assumption of a large number of floes, we develop the Vlasov-type kinetic model, followed by further development for large scales of the dynamics captured by a hydrodynamic model. In both kinetic and hydrodynamical models, we study the asymptotic behavior of total momentum and energy. In Section \ref{sec:num}, we present several numerical simulations in various settings to verify our theoretical findings for the particle model.  Finally, Section \ref{sec:conclusion} is devoted to a brief summary of main results and some remaining issues for future work.

\section{Particle description for ice floes} \label{sec:pmodel}
In this section, we first present an idealized particle model for sea ice floes following recent works \cite{chen2022superfloe, chen2021lagrangian, damsgaard2018application, herman2016discrete}, and then we introduce the LaSalle invariance principle, which serves as a tool for the establishment of an asymptotic result on floe velocities. 
Our main result is the relaxation of floe velocities to ocean velocity (see Theorem \eqref{thm:v0}) in Section \ref{sec:pmodelab}, which establishes the dynamics of the total momentum and energies of the particles in the entire system.

\subsection{The particle model} \label{sec:dem}
Consider non-rotating and colliding ice floes with the geometry of cylinders. 
Given a system of $n$ floes, 
we denote by $r^i$ the radius and $h^i$ the thickness (or height) of the $i$-th floe with $i \in [n] := \{1,2, \cdots, n \}$. 
The radius and thickness characterize the floe size. 
In a realistic sea ice floe setting (particularly in marginal ice zones), the floe size follows a power law distribution \cite{stern2018seasonal}, while the floe thickness distribution usually follows a Gamma distribution (see, for example, \cite{bourke1987sea} for the Arctic region and \cite{toyota2011size} for the Antarctic region). 
The mass of the $i$-th floe is $m^i=\rho_{ice}\pi (r^i)^2 h^i$, where the constant $\rho_{ice}$ is the density of sea ice floes. The floe position is denoted by $\bfx^i=(x^i,y^i)^T$ and the floe velocity is $\bfv^i=(u^i,v^i)^T$. 
Lastly, let $\bfu = \bfu(x,y)$ be the given ocean surface velocity.
The governing equations of non-rotating-colliding sea ice floe dynamics are given by Newton's equations:
\begin{align} 
\begin{aligned} \label{eq:dem}
  \frac{\d\bfx^i}{\d t} &= \bfv^i, \quad i \in [n], \\
  m^i\frac{\d\bfv^i}{\d t} & = \frac{1}{n} \sum_{j=1}^{n} \bff_c^{ij}  + \alpha^i\left(\bfu^i -\bfv^i\right)\left| \bfu^i -\bfv^i\right| =: \bfF^i, 
\end{aligned}
\end{align}  
where $|\bfu^i -\bfv^i|$ is the $\ell^2$-norm of $\bfu^i -\bfv^i$ in \rev{${\mathbb R}^2$},  $\alpha^i  = \pi \rho_o (2 C_{v,o} r^i \cdot D^i + C_{h,o}(r^i)^2)$ is the drag coefficient and $\bff_c^{ij}$ (in default, $i\ne j$; we assume $\bff_c^{ii}=0$ for notational simplicity) specifies the contact force. 
$\rho_o$ is ocean density, $D^i$ is the ice-floe draft (the part of the ice floe below the water surface), $C_{v,o}$ is ocean vertical drag coefficients, and $C_{h,o}$ is ocean horizontal drag coefficients. By default, we set $D^i = 0.9 h^i.$
The equations \eqref{eq:dem} describe the dynamics of each flow particle. 
They are the equations of motion following Newton’s law.
We scale the summation of the contact force by the factor of $1/n$ to apply the mean-field theory for deriving the kinetic and hydrodynamic models.
For a system with a fixed number of floes, this scaling factor is understood as a factor of the mass and drag coefficient.
Herein, for simplicity, we consider non-rotating ice floes and assume normal contact forces.
The floe contact force is nonzero when two floes are in contact, i.e., when
$$
 \delta^{ij} := |\bfx^i-\bfx^j| - (r^i+r^j) < 0, \quad i\ne j, \quad i,j \in [n].
$$
We follow the Hertz contact theory \cite{hertz1882ueber,puttock1969elastic} and adopt a simple model for the normal contact force (see the supplementary material of \cite{herman2016discrete}):
\begin{align} \label{eq:cfn}
    \bff_c^{ij} := \Big( \kappa^{ij}_1 \delta^{ij} + \kappa^{ij}_2 (\bfv^i  - \bfv^j )\cdot\bfn^{ij} \Big) \bfn^{ij},  \quad \kappa^{ij}_1 := \frac{\pi}{4} E_e h^{ij}_e, \quad \kappa^{ij}_2 := \beta \sqrt{ \frac{5\pi}{4} E_e h^{ij}_e m^{ij}_e}, 
\end{align}
where $\bfn^{ij} :=  \frac{\bfx^j-\bfx^i}{|\bfx^j-\bfx^i|}$ is the unit normal vector pointing from the center of floe $i$ to floe $j$, $E_e$ is the effective contact modulus, $h^{ij}_e := \min\{h^i, h^j\}$ is the effective contact thickness (part of the thickness in contact), $m^{ij}_e := \frac{m^i m^j}{m^i + m^j}$ is the effective mass, and
\begin{equation} \label{eq:beta}
\beta := \frac{\ln e_r}{\sqrt{\ln^2 e_r + \pi^2} } < 0.
\end{equation}
Here $e_r$ is the restitution coefficient in $[0, 1)$ and system parameters $\kappa^{ij}_1, \kappa^{ij}_2, h^{ij}_e$ and $m^{ij}_e$ depend on the pairs of floe $i$ and $j$ and $E_e$ is constant for all pairs. 
 Details are referred to equations (1), (8), (20), and (26) in the supplementary material of \cite{herman2016discrete}.

\begin{remark} 
The floe contact model is, in general, more complicated than \eqref{eq:cfn}. 
Herein, we consider a simple model to initiate the derivation of kinetic and hydrodynamic descriptions. 
The restitution coefficient for ice floes is the measure of kinetic energy loss during floe collisions. The coefficient is usually a positive real number between 0 and 1. A value of 0 indicates a perfectly inelastic collision, while a value of 1 indicates a perfectly elastic collision. The typical values of $e_r$  for sea ice are between 0.1 and 0.3; see \cite{li2020laboratory}. In ice floe studies, values between 0 and 1 are often employed; see, for example \cite{herman2019wave}.
\end{remark}

\begin{remark} 
\rev{
The present model \eqref{eq:dem} is an idealized floe particle model that captures the main feature of floe--floe collisions and ocean drag forces. Environmental effects can be included. For example, atmospheric wind forcing can be modeled as an additional drag term acting on the floes, with direction and magnitude determined by wind fields \cite{damsgaard2018application, herman2016discrete}. Thermodynamic processes such as melting or freezing can be incorporated through source terms that modify floe thickness and size distributions \cite{thorndike1975thickness,toppaladoddi2015theory,roach2018emergent}. These extensions would enable the model to capture the influence of changing climatic conditions on floe evolution and sea ice rheology. Herein, including these extensions will make the model much more complicated, and we consider this idealized model for the simplicity of the rigorous development of particle-to-continuum models. The development of the extensions is subject to future work.
}
\end{remark}

\subsection{The LaSalle invariance principle} 
In this subsection, we recall the LaSalle invariance principle \cite{khalil2002nonlinear,haddad2008nonlinear} to be used in later sections.  
Consider the following Cauchy problem for the continuous dynamical system on $\bbr^d$:
\begin{equation} \label{Cauchy}
\begin{cases}
{\dot \bfx} = \boldsymbol{F}(\bfx), \quad t \in {\mathbb R}_+ :=(0, \infty), \\
\bfx \Big|_{t = 0} = \bfx_0.
\end{cases}
\end{equation}
\rev{Herein, we abuse the notation $\bfx \in \mathbb{R}^d$ for generality.} In the sequel, we denote the solution for the Cauchy problem \eqref{Cauchy} by $\varphi_t(\bfx_0)$, and we call $\{ \varphi_t(\bfx_0) \}$ the orbit issued from $\bfx_0$. Next, we introduce the definition of $\omega$-limit set and recall its basic properties.
\begin{definition} 
The $\omega$-limit set of $\bfx_0$, denoted by $\omega(\bfx_0)$, is defined as the collection of all limit(accumulation) points of the orbit $\{ \varphi_t(\bfx_0) \}_{t \geq 0}$:
	\[\omega(\bfx_0):=\left\{\bfx \in \bbr^d:~\exists~\{t_n\}_{n\geq 1}\quad\mbox{such that}\quad\lim_{n \to \infty}t_n=\infty\quad\mbox{and}\quad \lim_{n \to \infty} \varphi_{t_n}(\bfx_0)=\bfx
	\right\}. \]
\end{definition}
Then, it is well known that the $\omega$-limit set is closed, positively invariant, and connected. If the orbit $\{ \varphi_t(\bfx_0) \}$ is uniformly bounded in $t$, then it must have a limit point by the Bolzano-Weierstrass theorem. Thus, $\omega(\bfx_0)$ is nonempty.  

Now, we are ready to state the LaSalle invariance principle, which asserts the asymptotic stability of an equilibrium point to \eqref{Cauchy}.  
\begin{proposition}\label{prop:lsip}
	Suppose that the vector field $\boldsymbol{F}$ and a nonempty open set $\Omega$ satisfy the following conditions:
	\begin{enumerate}
		\item
		$\boldsymbol{F}$ is locally Lipschitz continuous vector field on $\bbr^d$.
		\item
		Let $L:\Omega \to \mathbb{R}$ be a continuously differentiable functional such that its orbital derivative along the flow generated by \eqref{Cauchy} is nonpositive:
		\[   {\dot L}(\bfx) =\nabla L \cdot F(\bfx) \leq 0 \quad \mbox{for all $\bfx \in \Omega$}. \]
		\item
		${\omega(\bfx)}$ is bounded and contained in $\Omega$.
	\end{enumerate}
	Then, the trajectory $\phi_t(\bfx_0)$ approaches to the largest invariant set contained in $\{\bfx \in \bbr^d:~\dot{L}(\bfx)=0 \}$.
\end{proposition}

\subsection{Asymptotic behavior} \label{sec:pmodelab}
In this subsection, we study the asymptotic behavior of the particle model  \eqref{eq:dem} using the energy estimate and the LaSalle invariance principle in Proposition \ref{prop:lsip}. We assume that the mass of each floe does not change over time (thus, no melting, freezing, or fracturing).  First, we define the floe moments: 
\begin{equation} \label{eq:Ms}
\begin{aligned}
    M_0 & = \sum_{i=1}^n m^i, \qquad \bfM_1 = \sum_{i=1}^n m^i \bfv^i, \qquad M_2 = M_{2,\bfv} + M_{2,\bfx}, \\
    M_{2,\bfv} &= \frac{1}{2} \sum_{i=1}^n m^i | \bfv^i|^2, \qquad M_{2,\bfx} = \frac{1}{4n}\sum_{i,j=1}^n \kappa^{ij}_1  (\delta^{ij})^2, 
\end{aligned}
\end{equation}
where $M_{2,\bfx}$,  $M_{2,\bfv}$, and  $M_2$ represent total normal strain energy (the fraction $\frac{1}{4}$ is due to the repetition), total translational kinetic energy,
and total energy, respectively. By definition, $\kappa^{ij}_1$ depends on the thickness of the floe pairs; thus, $\kappa^{ij}_1$ is not a uniform constant and is not factored out of the summation in $M_{2,\bfx}$.

\begin{lemma} \label{lem:sumparticle}
\emph{(Propagation of velocity moments)} 
Let $(\bfx^i, \bfv^i)$ be a global solution to the system \eqref{eq:dem}. Then, the following assertions hold.
\begin{enumerate}
\item
The first and second velocity moments satisfy 
\begin{align}
\begin{aligned} \label{eq:M1M2}
&(i)~\frac{\d\bfM_1}{\d t}  = \sum_{i=1}^{n} \alpha^i\left(\bfu^i -\bfv^i\right)\left| \bfu^i -\bfv^i\right|. \\
&(ii)~\frac{\d M_2}{\d t}  =  \frac{1}{2n} \sum_{i, j=1}^n \kappa^{ij}_2 |(\bfv^i - \bfv^j )\cdot\bfn^{ij} |^2 +  \sum_{i=1}^{n} \alpha^i\left(\bfu^i -\bfv^i\right)\left| \bfu^i -\bfv^i\right| \cdot \bfv^i.
\end{aligned}
\end{align}
\item
If $\alpha^i  = 0, \forall~i \in [n],$  then there exist positive constants $A_0$ and $A_1$ such that 
\[ M_2(t)  \ge M_{2}(0) e^{-A_0 t} + \frac{A_1}{A_0} |\bfM_1(0)|^2  \Big( 1 - e^{-A_0 t} \Big). \]
\end{enumerate}
\end{lemma}
\begin{proof}
(i)~For the first assertion, we first note that the relations 
\[ \delta^{ij} = \delta^{ji}, \quad  \bfn^{ij} = -\bfn^{ji}, \quad i, j \in [n], \quad i\ne j\]
to see the anti-symmetry of floe-floe contact force:
\begin{equation} \label{eq:cfij}
\bff_c^{ij} = -\bff_c^{ji}, \quad i, j \in [n].
 \end{equation}
Now, we use \eqref{eq:dem},  \eqref{eq:Ms}, and  \eqref{eq:cfij} to find 
\begin{align*}
\begin{aligned}
    \frac{\d\bfM_1}{\d t} & = \frac{\d}{\d t} \sum_{i=1}^n m^i \bfv^i  = \sum_{i=1}^n m^i \frac{\d \bfv^i}{\d t}  =  \sum_{i=1}^n \Big( \frac{1}{n} \sum_{j=1}^{n} \bff_c^{ij} + \alpha^i\left(\bfu^i -\bfv^i\right)\left| \bfu^i -\bfv^i\right| \Big) \\
    &= \frac{1}{n} \sum_{i, j=1}^{n} \bff_c^{ij}  + \sum_{i=1}^{n} \alpha^i\left(\bfu^i -\bfv^i\right)\left| \bfu^i -\bfv^i\right| =\sum_{i=1}^{n} \alpha^i\left(\bfu^i -\bfv^i\right)\left| \bfu^i -\bfv^i\right|. 
\end{aligned}
\end{align*}
For the second assertion, we take an inner product of $ \eqref{eq:dem}_2$ with $\bfv^i$ and sum up the resulting equations over all $i \in [n]$ to get 
\begin{align}
\begin{aligned} \label{eq:p-1}
    \frac{\d M_{2,\bfv}}{\d t} & = \sum_{i=1}^n m^i\frac{\d\bfv^i}{\d t} \cdot \bfv^i =  \sum_{i=1}^n \Big( \frac{1}{n}\sum_{j=1}^{n} \bff_c^{ij}  + \alpha^i\left(\bfu^i -\bfv^i\right)\left| \bfu^i -\bfv^i\right| \Big) \cdot \bfv^i   \\
    & = \frac{1}{n}  \sum_{i, j=1}^n  \Big( \kappa^{ij}_1 \delta^{ij} + \kappa_2 (\bfv^i  - \bfv^j )\cdot\bfn^{ij} \Big) \bfn^{ij} \cdot \bfv^i  + \sum_{i=1}^{n} \alpha^i\left(\bfu^i -\bfv^i\right)\left| \bfu^i -\bfv^i\right| \cdot \bfv^i  \\
    & =:  {\mathcal I}_{11} + {\mathcal I}_{12} +   \sum_{i=1}^{n} \alpha^i\left(\bfu^i -\bfv^i\right)\left| \bfu^i -\bfv^i\right| \cdot \bfv^i.
\end{aligned}
\end{align}
Below, we estimate the terms ${\mathcal I}_{1i},.~i=1,2$ one by one. \newline

\noindent $\bullet$~Case A.1:~We use $ \eqref{eq:dem}_1$ and the relations:
\[ \delta^{ij} = \delta^{ji}, \quad  \bfn^{ij} = -\bfn^{ji}  \]
to rewrite
\begin{align}
\begin{aligned} \label{eq:p-2}
    {\mathcal I}_{11} & =  \frac{1}{n} \sum_{i, j=1}^n \kappa^{ij}_1 \delta^{ij}\bfn^{ij} \cdot \frac{\d\bfx^i}{\d t}  = \frac{1}{n} \sum_{i, j=1}^n \kappa^{ij}_1 \Big(|\bfx^i-\bfx^j|- (r^i+r^j)\Big)  \bfn^{ij} \cdot \frac{\d\bfx^i}{\d t}   \\
    & = - \frac{1}{n} \sum_{i, j=1}^n \kappa^{ij}_1 \Big(|\bfx^i-\bfx^j|- (r^i+r^j)\Big)  \bfn^{ij} \cdot \frac{\d\bfx^j}{\d t}   \\
    & = - \frac{1}{2n} \sum_{i, j=1}^n \kappa^{ij}_1 \Big(|\bfx^i-\bfx^j|- (r^i+r^j) \Big)  \bfn^{ij} \cdot \frac{\d (\bfx^j - \bfx^i)}{\d t}  \\
    & = - \frac{1}{2n} \sum_{i, j=1}^n \kappa^{ij}_1 \delta^{ij} \cdot \frac{\d \delta^{ij}}{\d t} = - \frac{\d M_{2,\bfx}}{\d t},   
\end{aligned}
\end{align}
where we used the following identity: 
\begin{align}
\begin{aligned} \label{eq:dddt}
    \frac{\d \delta^{ij}}{\d t} & = \frac{\d}{\d t} \big(|\bfx^i-\bfx^j|- (r^i+r^j)\big)  = \frac{\d}{\d t}  \sqrt{(x^j - x^i)^2 + (y^j - y^i)^2}  \\
    & = \frac{2(x^j - x^i) \d (x^j - x^i) + 2(y^j - y^i) \d (y^j - y^i) }{2 \sqrt{(x^j - x^i)^2 + (y^j - y^i)^2} \d t} = \bfn^{ij} \cdot \frac{\d (\bfx^j - \bfx^i)}{\d t}.   
\end{aligned}
\end{align}

\vspace{0.2cm}

\noindent $\bullet$~Case A.2:  Similarly, we have
\begin{align}
\begin{aligned} \label{eq:p-3}
    {\mathcal I}_{12} &  =  \frac{1}{n} \sum_{i, j=1}^n \kappa^{ij}_2  \Big(( (\bfv^i - \bfv^j )\cdot\bfn^{ij}) \bfn^{ij}  \Big) \cdot \bfv^i = - \frac{1}{n} \sum_{i, j=1}^n \kappa^{ij}_2  \Big( ( (\bfv^i - \bfv^j )\cdot\bfn^{ij} ) \bfn^{ij} \Big) \cdot \bfv^j   \\
    & =  \frac{1}{2n}  \sum_{i, j=1}^n \kappa^{ij}_2  \Big( ((\bfv^i - \bfv^j )\cdot\bfn^{ij} \bfn^{ij} ) \Big) \cdot (\bfv^i - \bfv^j ) =  \frac{1}{2n} \sum_{i, j=1}^n \kappa^{ij}_2   \Big( (\bfv^i - \bfv^j )\cdot\bfn^{ij} \Big)^2. 
\end{aligned}   
\end{align}
In \eqref{eq:p-1}, we combine \eqref{eq:p-2} and \eqref{eq:p-3} to arrive at the desired estimates. 

\vspace{0.2cm}

\noindent (ii)~Suppose that 
\[  \alpha^i=0, \quad  \forall~i \in [n]. \]
Then, the relation $\eqref{eq:M1M2}$ reduces to 
\begin{equation} \label{eq:p-3-1} 
\frac{\d\bfM_1}{\d t} = 0, \quad \frac{\d M_2}{\d t}  =  \frac{1}{2n} \sum_{i, j=1}^n \kappa^{ij}_2  \big( (\bfv^j - \bfv^i )\cdot\bfn^{ij} \big)^2.  
\end{equation}
On the other hand, by $\eqref{eq:cfn}$ and $\beta<0$ as in \eqref{eq:beta}, we have 
\[ \kappa^{ij}_2 = \beta \sqrt{ \frac{5\pi}{4} E_e h^{ij}_e m^{ij}_e} <0. \]
Together with \eqref{eq:p-3-1}, this implies
\[
\frac{\d M_2}{\d t} \leq 0, \quad \mbox{i.e., total energy is dissipative.}
\]
Now, we use the Cauchy–Schwarz inequality and the boundedness of the mass of floe particles to see that there exists a positive $A_0$ as the floe particle radii, thickness, and masses are fixed for a fixed system with $n$ floes (recalling $\kappa^{ij}_2<0$ for all pairs of colliding floes) such that 
\begin{align}
\begin{aligned} \label{eq:p-3-2}
\frac{\d M_2}{\d t}  &= -\frac{1}{2n} \sum_{i, j=1}^n |\kappa^{ij}_2|  \big( (\bfv^j - \bfv^i )\cdot\bfn^{ij} \big)^2 \\
& \geq - \frac{1}{2n} \Big( \max_{i,j}  |\kappa^{ij}_2| \Big)  \sum_{i, j=1}^n |\bfv^j - \bfv^i |^2 \\
& = - \frac{1}{2n} \Big( \max_{i,j}  |\kappa^{ij}_2| \Big)  \sum_{i, j=1}^n \Big( |\bfv^j |^2 + |\bfv^i |^2 - 2  \bfv^i \cdot \bfv^j \Big)  \\
&= - \frac{1}{2n} \Big( \max_{i,j}  |\kappa^{ij}_2| \Big) \Big( 2n  \sum_{i=1}^n  |\bfv^i |^2 - 2 \Big| \sum_{i=1}^n   \bfv^i \Big|^2 \Big )  \\
&= -\Big( \max_{i,j}  |\kappa^{ij}_2| \Big)  \sum_{i=1}^n  |\bfv^i |^2 +   \frac{1}{n} \Big( \max_{i,j}  |\kappa^{ij}_2| \Big) \Big| \sum_{i=1}^n   \bfv^i \Big|^2  \\
& = -\Big( \max_{i,j}  |\kappa^{ij}_2| \Big) {\mathcal I}_{21} +   \frac{1}{n} \Big( \max_{i,j}  |\kappa^{ij}_2| \Big) {\mathcal I}_{22}.
\end{aligned}
\end{align}
Next, we estimate the terms ${\mathcal I}_{2i}, i=1,2,$ one by one. \newline

\noindent $\bullet$~(Estimate of ${\mathcal I}_{21}$): Note that 
\begin{align}
\begin{aligned} \label{eq:p-3-3}
{\mathcal I}_{21} &=  \sum_{i=1}^n  |\bfv^i |^2 =  \sum_{i=1}^n \frac{m^i}{m^i}  |\bfv^i |^2   \leq \frac{2}{\min_{i} m^i} \Big( \frac{1}{2} \sum_{i=1}^n  m^i  |\bfv^i |^2 \Big)  = \frac{2}{\min_{i} m^i} M_{2,\bfv} \\
& \leq \frac{2}{\min_{i} m^i} M_{2}.
\end{aligned}
\end{align}
\vspace{0.1cm}

\noindent $\bullet$~(Estimate of ${\mathcal I}_{22}$): We use \eqref{eq:p-3-1} to see
\begin{equation} \label{eq:p-3-4}
 {\mathcal I}_{22} = \Big| \sum_{i=1}^n  \frac{m^i}{m^i}   \bfv^i \Big|^2 \geq \frac{1}{M_0^2} \Big| \sum_{i=1}^n m^i  \bfv^i  \Big|^2 = \frac{1}{M_0^2}  |\bfM_1(0)|^2.
\end{equation}
We combine \eqref{eq:p-3-2}, \eqref{eq:p-3-3} and \eqref{eq:p-3-4} to obtain
\begin{align*}
\begin{aligned}
\frac{\d M_2}{\d t} &\geq  -\frac{2\Big( \max_{i,j}  |\kappa^{ij}_2| \Big)}{\min_{i} m^i} M_{2} + \frac{1}{n M_0^2} \Big( \max_{i,j}  |\kappa^{ij}_2| \Big)  |\bfM_1(0)|^2  \\
&=: -A_0 M_2 + A_1 |\bfM_1(0)|^2.
\end{aligned}
\end{align*}
\rev{Then,} Grönwall's lemma \cite{gronwall1919note} yields the desired estimate. 
\end{proof}

\begin{remark}
    Similar to the result \cite{ha2008Kinetic} for the Cucker-Smale particle model, if the drag force is absent (i.e., $\alpha^i=0$), Proposition \ref{lem:sumparticle} tells that the total energy $M_2$ is monotonically decreasing with a lower bound. 
\end{remark}

\begin{theorem} \label{thm:v0} Suppose that the given ocean surface velocity is constant:
\[  \bfu^i = \bfu^{\infty}:~\mbox{constant}, \quad \forall~i \in [n], \]
and let $(\bfx^i, \bfv^i)$ be a global solution to the system \eqref{eq:dem}. Then, we have
\[ \lim_{t \to \infty} \|  \bfv^i  -  \bfu^{\infty} \| =0, \quad \forall~i \in [n].  \]
\end{theorem}
\begin{proof} Due to the translational invariance of system \eqref{eq:dem}, we may assume that 
\[  \bfu^{\infty} \equiv 0. \]
In this setting, the system \eqref{eq:dem} becomes 
\begin{equation*} 
\begin{cases} 
\displaystyle  \frac{\d\bfx^i}{\d t} = \bfv^i, \quad i \in [n], \\
 \displaystyle m^i\frac{\d\bfv^i}{\d t} =  \frac{1}{n} \sum_{j=1}^{n}\bff_c^{ij} - \alpha^i \bfv^i |\bfv^i|. 
\end{cases}
\end{equation*}  
Then, we claim that
\begin{equation} \label{eq:p-5}
\lim_{t \to \infty} \bfv^i(t) = 0.
\end{equation}
{\it Proof of \eqref{eq:p-5}}:~It follows from Lemma \ref{lem:sumparticle} that 
\[ \frac{\d M_2}{\d t}  =  \frac{1}{2n} \sum_{i, j=1}^n \kappa^{ij}_2  \big( (\bfv^i - \bfv^j )\cdot\bfn^{ij} \big)^2 - \sum_{i=1}^{n} \alpha^i \left| \bfv^i\right|^3. \]
Note that for all pairs of colliding floes, $\kappa^{ij}_2 < 0.$
Thus, we have
\begin{align*}
\begin{aligned}
 \frac{\d M_2}{\d t}  = 0 \quad &\iff \quad  (\bfv^j - \bfv^i )\cdot\bfn^{ij} = 0 \quad \mbox{and} \quad  \alpha^i \left| \bfv^i\right|^3 = 0, \quad \forall~i, j \in [n] \\
                                        & \iff \quad \bfx^i:~\mbox{arbitrary}, \quad  \bfv^i = 0 \quad \forall~i \in [n].
\end{aligned}
\end{align*}
Hence we have
\[
S = \Big \{ (\bfx_1, \ldots, \bfx_n, \bfv_1, \ldots, \bfv_n) \in \rev{\bbr^{4n}}:~{\dot M}_2 = 0 \Big \} =  \Big \{ (\bfx_1, \ldots, \bfx_n, 0, \ldots, 0)   \Big \}.
\]
Now, we take $M_2$ as the Lyapunov functional for Proposition \ref{prop:lsip} to see that the flow converges to the largest invariant subset of $S$. Thus, we have the desired estimate. 
\end{proof}

\section{From particle to kinetic description} \label{sec:kmodel}
In this section, we first recall the formal derivation from the particle model to the kinetic model as the mean-field approximation of the particle model with $n \gg 1$. 
We assume that the number of particles involved in the particle system \eqref{eq:dem} is sufficiently large so that it becomes meaningful to use the mean-field approximation via the one-particle distribution function to describe the overall effective dynamics of the original system.

\subsection{Kinetic model for ice floe dynamics}
We first rewrite the floe particle model \eqref{eq:dem} as 
\begin{equation}
\begin{cases} \label{eq:p-6}
 \displaystyle \frac{\d\bfx^i}{\d t} = \bfv^i, \quad i \in [n], \\
 \displaystyle \frac{\d\bfv^i}{\d t} = \frac{1}{m^i} \Big[  \frac{1}{n} \sum_{j=1}^{n}\bff_c^{ij}  + \alpha^i\left(\bfu^i -\bfv^i\right)\left| \bfu^i -\bfv^i\right| \Big] =: \frac{\bfF^i}{m^i}, \\
 \displaystyle \frac{dr^i}{dt} = 0, \quad \frac{dh^i}{dt} = 0,
\end{cases} 
\end{equation}
where we assume that the floe sizes and thicknesses do not change in time. In what follows, we adopt the method of the BBGKY hierarchy to derive a kinetic equation for the one-particle distribution function over the generalized phase space $\bbr_{\bfx}^2 \times \bbr_{\bfv}^2 \times \bbr_+ \times \bbr_+$. \newline

\noindent $\bullet$~Step A (Derivation of the Liouville equation for $n$-particle distribution function):  First, we define $n$-particle distribution function over the $n$-particle phase space $\bbr^{4n} \times \bbr_+^{2n}$:
\begin{equation} \label{eq:p-7}
 F^n = F^n(t,\bfx^1, \bfv^1, r^1, h^1, \cdots, \bfx^n, \bfv^n, r^n, h^n), 
\end{equation}
for $(\bfx^i, \bfv^i, r^i, h^i) \in \mathbb{R}^2\times \mathbb{R}^2 \times \mathbb{R}_+ \times \mathbb{R}_+$. \newline

Note that the $n$-particle probability density function $F^n$ is symmetric in its phase variable in the sense that 
\begin{align}
\begin{aligned} \label{eq:p-8}
&  F^n(t,\cdots, \bfx^i, \bfv^i, r^i, h^i, \cdots, \bfx^j, \bfv^j, r^j, h^j, \cdots) \\
 & \hspace{1.5cm}   = F^n(t, \cdots, \bfx^j, \bfv^j, r^j, h^j, \cdots, \bfx^i, \bfv^i, r^i, h^i, \cdots).
\end{aligned}
\end{align}
Then, $F^n$ satisfies the Liouville equation on the generalized $n$-particle phase space:
\begin{equation} \label{eq:p-9}
\partial_t F^n + \sum_{i=1}^{n} \nabla_{\bfx^i} \cdot ( \dot{\bfx}^i F^n) + \sum_{i=1}^{n} \nabla_{\bfv^i} \cdot ( \dot{\bfv}^i F^n)  + \sum_{i=1}^{n} \partial_{r^i} ( \dot{r}^i F^n)  + \sum_{i=1}^{n} \partial_{h^i} ( \dot{h}^i F^n)  = 0,
\end{equation}
where $ \nabla_{\bfx^i}  \cdot (~  ~)$ and $ \nabla_{\bfv^i}  \cdot (~  ~)$ denote the divergences in $\bfx^i$ and $\bfv^i$-variables, respectively. 
By using \eqref{eq:p-6}, the above system can \rev{be also} written as 
\begin{equation} \label{eq:p-10}
\partial_t F^n +  \sum_{i=1}^{n} \bfv^i \cdot \partial_{\bfx^i} F^n + \sum_{i=1}^{n} \nabla_{\bfv^i} \cdot \Big ( \frac{\bf F^i}{m^i} F^n \Big ) = 0.
\end{equation}
\vspace{0.2cm}

\noindent $\bullet$~Step B (Derivation of equation  for $j$-particle distribution function):  For notational simplicity, we set 
\[ \bfz := (\bfx,\bfv,r,h), \quad   d\bfz^j = d\bfx^j d\bfv^j dr^j dh^j, \quad j \in [n], \quad  \d E^{n:j} = \prod_{i=j+1}^n d\bfz^i. \]
Then, we introduce $j$-marginal distribution function $F^{n:j}$ by the integration of \eqref{eq:p-7}:
\[ F^{n:j}(t, \bfx_1, \ldots, \bfx_j, \bfv_1, \ldots, \bfv_j, r^1, \ldots, r^j, h^1, \ldots, h^j)  := \int_{\bbr^{4(n-j)} \times \bbr_+^{2(n-j)}} F^{n} \d E^{n:j}. \]
Next, we derive an equation for $ F^{n:j}$. For this,  we rewrite \eqref{eq:p-10} as follows.
\begin{align}
\begin{aligned} \label{eq:p-11}
 \partial_t F^n &= -\sum_{i=1}^j   \nabla_{\bfx^i} \cdot  (\bfv^i F^n) - \sum_{i=1}^{j} \nabla_{\bfv^i} \cdot \Big ( \frac{\bf F^i}{m^i} F^n \Big ) - \sum_{i=j+1}^{n}  \nabla_{\bfx^i} \cdot ( \bfv^j F^n)   \\
& \hspace{0.5cm} - \sum_{i=j+1}^{n} \nabla_{\bfv^i} \cdot \Big ( \frac{\bf F^i}{m^i} F^n \Big ).
\end{aligned}
\end{align}
Now, we integrate the Liouville equation \eqref{eq:p-11} over
\[ (\bfx^{j+1}, \cdots, \bfx^{n}, \bfv^{j+1}, \cdots, \bfv^{n}, r^{j+1},. \cdots, r^n, h^{j+1}, \cdots, h^n) \]
to obtain
\begin{align}
\begin{aligned} \label{eq:p-12}
\partial_t F^{n:j} &=  -\sum_{i=1}^j    \int_{\bbr^{4(n-j)} \times \bbr_+^{2(n-j)}}  \nabla_{\bfx^i} \cdot  (\bfv^i F^n) \d E^{n:j} \\
& - \sum_{i=1}^{j} \int_{\bbr^{4(n-j)} \times \bbr_+^{2(n-j)}} \nabla_{\bfv^i} \cdot \Big ( \frac{\bf F^i}{m^i} F^n \Big ) \d E^{n:j} \\
& -  \sum_{i=j+1}^{n}  \int_{\bbr^{4(n-j)} \times \bbr_+^{2(n-j)}}  \nabla_{\bfx^i} \cdot ( \bfv^j F^n) \d E^{n:j} \\
&- \sum_{i=j+1}^{n} \int_{\bbr^{4(n-j)} \times \bbr_+^{2(n-j)}} \nabla_{\bfv^i} \cdot \Big ( \frac{\bf F^i}{m^i} F^n \Big ) \d E^{n:j} \\
&=: {\mathcal I}_{21} + {\mathcal I}_{22} + {\mathcal I}_{23} + {\mathcal I}_{24}.
\end{aligned}
\end{align}
In the next lemma, we estimate the terms ${\mathcal I}_{2i},~i=1,\cdots, 4$, one by one. 
\begin{lemma} \label{lem:I2}
Let $F^n$ be a global solution to \eqref{eq:p-9} which decays to zero sufficiently fast at $|\bfx^i| = \infty$ and $|\bfv^i| = \infty$.  Then, we have  the following estimates:
\begin{align*}
\begin{aligned}
& (i)~{\mathcal I}_{21} = -  \sum_{i=1}^j  \bfv^i  \cdot \nabla_{\bfx^i} F^{n:j},  \quad {\mathcal I}_{23} = 0, \quad {\mathcal I}_{24}  = 0. \\
& (ii)~{\mathcal I}_{22} =  -\frac{1}{n} \sum_{i=1}^{j}  \nabla_{\bfv^i} \cdot  \Big[ \Big( \frac{1}{m^i}  \sum_{k=1}^{j} \bff_c^{ik} \Big)  F^{n:j} \Big] \\
& \hspace{1.5cm} -\frac{n-j}{n} \sum_{i=1}^{j}  \int_{\bbr^{4} \times \bbr_+^{2}} \nabla_{\bfv^i} \cdot \Big( \frac{1}{m^i} \bff_c^{i(j+1)} F^{n:j+1} \Big) d\bfz^{j+1} \\
& \hspace{1.5cm}  - \sum_{i=1}^{j} \frac{1}{m^i}  \nabla_{\bfv^i} \cdot \Big( \alpha^i\left(\bfu^i -\bfv^i\right)\left| \bfu^i -\bfv^i\right| F^{n:j} \Big).
\end{aligned} 
\end{align*}
\end{lemma}
\begin{proof}
\noindent (i) It is easy to check that 
\begin{align*}
\begin{aligned}
{\mathcal I}_{21} =  -\sum_{i=1}^j  \int_{\bbr^{4(n-j)} \times \bbr_+^{2(n-j)}}  \nabla_{\bfx^i} \cdot  (\bfv^i F^n) \d E^{n:j} =  -\sum_{i=1}^j  \nabla_{\bfx^i} \cdot (\bfv^i F^{n:j}).  
\end{aligned}
\end{align*}
On the other hand, we use the divergence theorem and decay condition of $F^n$ at infinity to find 
\[  {\mathcal I}_{23} = 0 \quad \mbox{and} \quad  {\mathcal I}_{24}  = 0. \]
(ii)~Now, we use \eqref{eq:p-8} to see that 
\begin{align}
\begin{aligned} \label{eq:p-13}
{\mathcal I}_{22} &= - \sum_{i=1}^{j} \int_{\bbr^{4(n-j)} \times \bbr_+^{2(n-j)}} \nabla_{\bfv^i} \cdot \Big ( \frac{\bf F^i}{m^i} F^n \Big ) \d E^{n:j} \\
     & = -\frac{1}{n} \sum_{i=1}^{j}  \int_{\bbr^{4(n-j)} \times \bbr_+^{2(n-j)}} \nabla_{\bfv^i} \cdot \Big( \frac{F^n }{m^i}  \sum_{k=1}^{n} \bff_c^{ik} \Big) \d E^{n:j} \\
    & \hspace{0.2cm} - \sum_{i=1}^{j}  \frac{1}{m^i} \int_{\bbr^{4(n-j)} \times \bbr_+^{2(n-j)}} \nabla_{\bfv^i} \cdot \Big( \alpha^i\left(\bfu^i -\bfv^i\right)\left| \bfu^i -\bfv^i\right| F^n \Big) \d E^{n:j} \\
    & =: {\mathcal I}_{221} + {\mathcal I}_{222}.
\end{aligned}
\end{align}
Below, we estimate the terms ${\mathcal I}_{22i}, i=1,2,$ one by one. \newline

\noindent $\diamond$~Case B.1 (Estimate of $ {\mathcal I}_{221}$): We use \eqref{eq:p-8} to find 
\begin{align}
\begin{aligned} \label{eq:p-14}
 {\mathcal I}_{221} &= -\frac{1}{n} \sum_{i=1}^{j}  \int_{\bbr^{4(n-j)} \times \bbr_+^{2(n-j)}} \nabla_{\bfv^i} \cdot \Big( \frac{F^n }{m^i}  \sum_{k=1}^{j} \bff_c^{ik} \Big) \d E^{n:j}  \\
  & \hspace{0.2cm} -\frac{1}{n} \sum_{i=1}^{j}  \int_{\bbr^{4(n-j)} \times \bbr_+^{2(n-j)}} \nabla_{\bfv^i} \cdot \Big( \frac{F^n }{m^i}  \sum_{k=j+1}^{n} \bff_c^{ik} \Big) \d E^{n:j}  \\
  &= -\frac{1}{n} \sum_{i=1}^{j}  \nabla_{\bfv^i} \cdot  \Big[ \Big( \frac{1}{m^i}  \sum_{k=1}^{j} \bff_c^{ik} \Big)  F^{n:j} \Big] \\
  &-\frac{n-j}{n} \sum_{i=1}^{j}  \int_{\bbr^{4} \times \bbr_+^{2}} \nabla_{\bfv^i} \cdot \Big( \frac{1}{m^i} \bff_c^{i(j+1)} F^{n:j+1} \Big) d\bfz^{j+1}.
\end{aligned}
\end{align}
\vspace{0.5cm}
\noindent $\diamond$~Case B.2 (Estimate of $ {\mathcal I}_{222}$): For $i \leq j$, since $\bfv^i$ is independent of $z^{j+1}, \cdots , z^n$, we have
\begin{align}
\begin{aligned} \label{eq:p-16}
{\mathcal I}_{222} &= - \sum_{i=1}^j  \frac{1}{m^i}  \int_{\bbr^{4(n-j)} \times \bbr_+^{2(n-j)}} \nabla_{\bfv^i} \cdot \Big( \alpha^i \left(\bfu^i -\bfv^i\right) \left | \bfu^i -\bfv^i \right | F^n \Big) \d E^{n:j}  \\
&= -\sum_{i=1}^{j} \frac{1}{m^i}   \nabla_{\bfv^i} \cdot \Big ( \alpha^i \left(\bfu^i  -\bfv^i \right)\left| \bfu^i -\bfv^i\right| F^{n:j} \Big).
\end{aligned}
\end{align}
In \eqref{eq:p-13}, we combine \eqref{eq:p-14} and \eqref{eq:p-16} to find the desired estimate. 
\end{proof}

\noindent In \eqref{eq:p-12}, we use Lemma \ref{lem:I2} to find 
\begin{align*}
\begin{aligned}
 \partial_t F^{n:j} &=  -  \sum_{i=1}^j  \bfv^i  \cdot \nabla_{\bfx^i} F^{n:j} -\frac{1}{n} \sum_{i=1}^{j}  \nabla_{\bfv^i} \cdot  \Big[ \Big( \frac{1}{m^i}  \sum_{k=1}^{j} \bff_c^{ik} \Big)  F^{n:j} \Big] \\
 &-\frac{n-j}{n} \sum_{i=1}^{j}  \int_{\bbr^{4} \times \bbr_+^{2}} \nabla_{\bfv^i} \cdot \Big( \frac{1}{m^i} \bff_c^{i(j+1)} F^{n:j+1} \Big) d\bfz^{j+1} \\
&- \sum_{i=1}^{j}  \frac{1}{m^i}  \nabla_{\bfv^i} \cdot \Big( \alpha^i\left(\bfu^i -\bfv^i\right)\left| \bfu^i -\bfv^i\right| F^{n:j} \Big).
\end{aligned}
\end{align*}
For a fixed $j$, let  $n \to \infty$ and we assume that the exists a limit $F^j$ such that 
\[ \lim_{n \to \infty} F^{n:j} = F^j \quad \mbox{in suitable sense}.  \]
Then, formally as $n \to \infty$, the limit $F^j$ satisfies 
\begin{align}
\begin{aligned} \label{eq:p-17}
 & \partial_t F^{j} +  \sum_{i=1}^j  \bfv^i  \cdot \nabla_{\bfx^i} F^{j}  + \sum_{i=1}^{j} \frac{1}{m^i}  \nabla_{\bfv^i} \cdot \Big( \alpha^i\left(\bfu^i -\bfv^i\right)\left| \bfu^i -\bfv^i\right| F^{j} \Big)  \\
 & \hspace{0.5cm} +\sum_{i=1}^{j}  \int_{\bbr^{4} \times \bbr_+^{2}} \nabla_{\bfv^i} \cdot \Big( \frac{1}{m^i} \bff_c^{i(j+1)} F^{j+1} \Big) d\bfz^{j+1}  
= 0.
\end{aligned}
\end{align}
Note that the dynamics of $F^j$ in \eqref{eq:p-17} depends on $F^{j+1}$.  In particular, for $j=1$, we remove the superscript for simplicity and set 
\[ F := F^1, \quad \bfx := \bfx^1, \quad \bfv := \bfv^1,  \quad \bfu =  \bfu^1, \quad \alpha := \alpha^1, \quad m:= m^1. \]
Then, it satisfies 
\begin{align}
\begin{aligned} \label{eq:p-18}
  \partial_t F +   \bfv  \cdot \nabla_{\bfx} F  +  \alpha \nabla_{\bfv} \cdot \Big( \frac{1}{m}  \left(\bfu -\bfv \right)\left| \bfu -\bfv \right| F \Big)  
+  \int_{\bbr^{4} \times \bbr_+^{2}} \nabla_{\bfv} \cdot \Big( \frac{1}{m} \bff_c^{12} F^{2} \Big) d\bfz^{2} 
= 0.
\end{aligned}
\end{align}

\noindent $\bullet$~Step C (Formal derivation of kinetic equation for one-particle distribution function):  We assume the ``molecular chaos assumption" \rev{(see, for example, \cite{spohn2012large})} by setting
\begin{equation} \label{eq:p-19}
F^2(t, \bfz, \bfz^*) = F(t, \bfz) \otimes F(t, \bfz^*). 
\end{equation}
Finally, we substitute the ansatz \eqref{eq:p-19} into \eqref{eq:p-18} to get the kinetic equation for $F$:
\begin{align}
\begin{aligned} \label{eq:p-20}
 & \partial_t F +   \bfv  \cdot \nabla_{\bfx} F  +  \frac{\alpha}{m} \nabla_{\bfv} \cdot \big( \left(\bfu -\bfv \right)\left| \bfu -\bfv \right| F \big)   \\
& \hspace{1.5cm} + \frac{1}{m}   \nabla_{\bfv} \cdot \Big[ \Big( \int_{\bbr^{4} \times \bbr_+^{2}}  \bff_c^{12} F(t, \bfz^*) d\bfz^*  \Big)  F(t, \bfz) \Big ] 
 = 0.
\end{aligned}
\end{align}
Note that 
\begin{align}
\begin{aligned} \label{eq:p-21}
    \bff_c^{12} & = \kappa_1 \chi(|\bfx^* - \bfx| - (r + r^*)) \bfn + \kappa_2 \big((\bfv  - \bfv^* )\cdot\bfn \big) \bfn, \\
        \alpha & = \pi \rho_o (2 C_{v,o} r \cdot D + C_{h,o} r^2), \quad m=\rho_{ice}\pi r^2 h,
\end{aligned}
\end{align}
where $\bfn = \bfn(\bfx, \bfx^*)=\frac{\bfx^*-\bfx }{|\bfx^*-\bfx|}$ for \rev{notational simplicity and $\chi$ is a function defined as}
\begin{equation} \label{eq:chi}
\chi(\xi) = 
\begin{cases}
\xi, \quad \text{when} \ \xi<0, \\
0, \quad \text{otherwise}.
\end{cases}
\end{equation}
In what follows, we use handy notation:
\begin{align}
\begin{aligned} \label{eq:snot}
& \gamma_o : =\gamma_{r,h, o} :=  \frac{ \alpha}{m}, 
\quad  \gamma_{\bfn} :=  \frac{\kappa_1}{m},  \quad \gamma_{\bfv} :=  \frac{\kappa_2}{m}, \quad \bff[F] = \bff_o +  \bff_{c}[F], \\
& \bff_{c}[F]  = \bff_{c,\bfn}[F] +  \bff_{c,\bfv}[F], \quad  \bff_o :=  \gamma_o \left(\bfu -\bfv \right)\left| \bfu -\bfv \right|, \\
& \bff_{c,\bfn}[F](t,\bfz) :=  \int_{\bbr^{4} \times \bbr_+^{2}}  \gamma_{\bfn}  \chi(|\bfx^* - \bfx| - (r + r^*)) \bfn(\bfx, \bfx^*)  F(t, \bfz^*) d\bfz^*, \\
& \bff_{c,\bfv}[F](t,\bfz) :=  \int_{\bbr^{4} \times \bbr_+^{2}}  \gamma_{\bfv} [ (\bfv  - \bfv^*) \cdot \bfn(\bfx, \bfx^*)] \bfn(\bfx, \bfx^*) F(t, \bfz^*) d\bfz^*.
\end{aligned}
\end{align}
Finally, we substitute \eqref{eq:p-21} and \eqref{eq:snot} into \eqref{eq:p-20} to find the Vlasov-McKean equation:
\begin{align}
\begin{aligned} \label{eq:p-23}
 \partial_t F +   \bfv  \cdot \nabla_{\bfx} F  +   \nabla_{\bfv} \cdot \Big[ \bff_o F \Big ] +  \nabla_{\bfv} \cdot \Big[ \bff_{c,\bfn}[F] F \Big ] +  \nabla_{\bfv} \cdot \Big[ \bff_{c,\bfv}[F] F \Big ]  = 0,
\end{aligned}
\end{align}
or equivalently,
\begin{equation} \label{eq:p-24}
 \partial_t F +   \bfv  \cdot \nabla_{\bfx} F +   \nabla_{\bfv} \cdot (\bff[F] F)= 0.
\end{equation}
Note that the equation \eqref{eq:p-24} is Galilean invariant.

\subsection{Macroscopic behavior of the kinetic model}
From now on, as long as there is no confusion, we suppress $t$-dependence in $F$: 
\[ F(\bfz) \equiv  F(t, \bfz), \quad \bfz \in \rev{ {\mathbb{R}^4 \times \bbr_+^2} } .  \]
\noindent Next, we define the energy functional as follows.
\begin{equation}  \label{eq:p-25}
\begin{cases} 
\displaystyle \E  := \E_K + \E_P, \quad \E_K := \frac{1}{2} \int_{\mathbb{R}^4 \times \bbr_+^2} |\bfv|^2 F(\bfz) \d \bfz, \\
\displaystyle \E_P  :=  \frac{1}{2} \int_{\mathbb{R}^8 \times \bbr_+^4} \gamma_{\bfn}  \Big( \int_0^{|\bfx^* - \bfx| - (r + r^*)} \chi(\eta) d\eta  \Big)  F(\bfz) F(\bfz^*) d\bfz d\bfz^*.
\end{cases}
\end{equation}
With the definition \eqref{eq:chi} in mind, the functional $\E_P$ becomes  
\begin{equation} \label{eq:Ep}
\displaystyle \E_P  :=  \frac{1}{4} \int_{\mathbb{R}^8 \times \bbr_+^4} \gamma_{\bfn}  \chi(|\bfx^* - \bfx| - (r + r^*))^2     F(\bfz) F(\bfz^*) d\bfz d\bfz^* \geq 0.
\end{equation}
Then, it is easy to see that component functionals in $\E$ are also nonnegative, hence the energy functional $\E$ is nonnegative. In the following two lemmas, we establish the temporal evolution of $\E_K$ and $\E_P$ as follows. 
\begin{lemma}[Macroscopic behavior] \label{lem:kmb}
Let $F$ be a global smooth solution to \eqref{eq:p-24} which decays \rev{to zero} sufficiently fast at infinity in phase space. Then, the following estimates hold.
\begin{align}
\begin{aligned} \label{eq:p-27}
& (i)~  \frac{\d}{\d t} \int_{\mathbb{R}^4 \times \bbr_+^2} F(t,\bfz) \d \bfz = 0. \\
& (ii)~\frac{\d}{\d t} \int_{\mathbb{R}^4 \times \bbr_+^2}  \bfv F \d \bfz = \int_{\mathbb{R}^4 \times \bbr_+^2} \gamma_o 
    \left(\bfu -\bfv\right)\left| \bfu -\bfv\right| F(t, \bfz)  \d \bfz. \\
& (iii)~\frac{\d}{\d t} \int_{\mathbb{R}^4} \frac{| \bfv |^2}{2} F(\bfz) \d \bfz =  \int_{\mathbb{R}^4 \times \bbr_+^2} \gamma_o \bfv \cdot  \left(\bfu -\bfv \right)\left| \bfu -\bfv \right| F(\bfz)  d \bfz  \\
& \hspace{1cm} + \frac{1}{2} \int_{\mathbb{R}^8 \times \bbr_+^4}   \gamma_{\bfn}  \chi(|\bfx^* - \bfx| - (r + r^*))( \bfv - \bfv^*)  \cdot \bfn(\bfx, \bfx^*)  F( \bfz^*) F(\bfz) d\bfz^* d \bfz \\
& \hspace{1cm} + \frac{1}{2}  \int_{\mathbb{R}^8 \times \bbr_+^4}   \gamma_{\bfv} [ (\bfv  - \bfv^*) \cdot \bfn(\bfx, \bfx^*)]^2  F(\bfz) F(\bfz^*) d\bfz^* d\bfz.
\end{aligned}
\end{align}  
\end{lemma}
\begin{proof} 
\noindent (i) We integrate \eqref{eq:p-24} over the phase space $\bbr^4 \times \bbr_+^2$ and use the decay of $F$ at infinity to find the conservation of total mass:
\[ \frac{\d}{\d t} \int_{\mathbb{R}^4 \times \bbr_+^2} F(t,\bfz) \d \bfz = 0. \]
\noindent (ii)~Next, we derive a balance law for momentum. For this, we use \eqref{eq:p-24} to see that 
\begin{equation} \label{eq:p-28}
\partial_t (\bfv F) +   \frac12 \nabla_{\bfx} \cdot \big (\bfv \otimes \bfv F \big)  +  \nabla_{\bfv} \cdot  \Big (  \bfv  \otimes \bff[F] F  \Big )=  \bff_o F +  \bff_{c}[F] F.
\end{equation}
Again, we integrate the relation \eqref{eq:p-28} over the phase space $\bbr^4 \times \bbr_+^2$ to find 
\begin{align}
\begin{aligned} \label{eq:p-29}
& \frac{\d}{\d t} \int_{\mathbb{R}^4 \times \bbr_+^2} \bfv F(t,\bfz) \d \bfz =  \int_{\mathbb{R}^4 \times \bbr_+^2}  \bff_o Fd\bfz + \int_{\mathbb{R}^4 \times \bbr_+^2} \bff_{c}[F] F d\bfz \\
&\hspace{1cm}  =: \int_{\mathbb{R}^4} \gamma_o 
    \left(\bfu -\bfv\right)\left| \bfu -\bfv\right| F(t, \bfz)  \d \bfx \d \bfv + {\mathcal I}_{31} + {\mathcal I}_{32}.
\end{aligned}
\end{align}
Below, we estimate the terms ${\mathcal I}_{3i}, i=1,2,$ one by one. \newline

\noindent $\bullet$~Case C.1 (Estimate of ${\mathcal I}_{31}$): We use the relation for $\bff_{c,\bfn}[F]$ in \eqref{eq:snot} and use the exchange symmetry $\bfz~\leftrightarrow~\bfz^*$ to see
\begin{align*}
\begin{aligned}
 {\mathcal I}_{31} &= \int_{\bbr^{8} \times \bbr_+^{4}}  \gamma_{\bfn}  \chi(|\bfx^* - \bfx| - (r + r^*)) \bfn(\bfx, \bfx^*)  F(\bfz^*) F(\bfz) d\bfz^* d\bfz \\
            &= -\int_{\bbr^{8} \times \bbr_+^{4}}  \gamma_{\bfn}  \chi(|\bfx^* - \bfx| - (r + r^*)) \bfn(\bfx, \bfx^*)  F(\bfz^*) F(\bfz) d\bfz^* d\bfz,
\end{aligned}
\end{align*}
where we used
\[  \bfn(\bfx^*, \bfx)  =  -\bfn(\bfx, \bfx^*).  \]
Hence, we have
\begin{equation} \label{eq:p-30}
{\mathcal I}_{31} = 0.
 \end{equation}
\vspace{0.2cm}

\noindent $\bullet$~Case C.2 (Estimate of ${\mathcal I}_{32}$): Similar to Case C.1, we have
\begin{align*}
\begin{aligned}
{\mathcal I}_{32} &=  \int_{\bbr^{8} \times \bbr_+^{4}}  \gamma_{\bfv} [ (\bfv  - \bfv^*) \cdot \bfn(\bfx, \bfx^*)] \bfn(\bfx, \bfx^*) F(\bfz^*) F(\bfz) d\bfz^* d\bfz \\
&=  -\int_{\bbr^{8} \times \bbr_+^{4}}  \gamma_{\bfv} [ (\bfv^*  - \bfv) \cdot \bfn(\bfx^*, \bfx)] \bfn(\bfx^*, \bfx) F(\bfz) F(\bfz^*) d\bfz d\bfz^*.
\end{aligned}
\end{align*}
Therefore, we have
\begin{equation} \label{eq:p-31}
{\mathcal I}_{32} = 0.
 \end{equation}
 In \eqref{eq:p-29}, we combine \eqref{eq:p-30} and \eqref{eq:p-31} to get the desired estimate. \newline
 
\noindent (iii)~We multiply $\frac{1}{2} |\bfv |^2$ to \eqref{eq:p-24} to get
\begin{align}
\begin{aligned} \label{eq:p-32}
& \partial_t \Big(\frac{1}{2} |\bfv |^2 F \Big) +   \nabla_{\bfx} \cdot \Big (\frac{1}{2} |\bfv |^2 \bfv F \Big)  +  \nabla_{\bfv} \cdot  \Big (\frac{1}{2} |\bfv |^2 \bff[F] F  \Big ) \\
& \hspace{2.5cm} = (\bfv \cdot \bff_o) F +  (\bfv \cdot \bff_{c}[F]) F.
\end{aligned}
\end{align}
Now, we integrate \eqref{eq:p-32}  over $\bbr^4 \times \bbr_+^2$ to find 
\begin{align}
\begin{aligned} \label{eq:p-32-1}
 \frac{\d}{\d t} \int_{\mathbb{R}^4 \times \bbr_+^2} \frac{1}{2} |\bfv |^2 F(\bfz) \d \bfz &  = \int_{\mathbb{R}^4 \times \bbr_+^2}  (\bfv \cdot \bff_o) F(\bfz)  d \bfz + \int_{\mathbb{R}^4 \times \bbr_+^2}   (\bfv \cdot \bff_{c}[F])(\bfz) F(t,\bfz)   d \bfz 
 \\
&  =:  {\mathcal I}_{41} + {\mathcal I}_{42}+ {\mathcal I}_{43}.
\end{aligned}
\end{align}
Below, we estimate the term ${\mathcal I}_{4i}, i=1,2,3,$ one by one. \newline

\noindent $\bullet$~Case D.1 (Estimate of ${\mathcal I}_{41}$):  We use the definition of $\bff_o$ in \eqref{eq:snot} to see
\begin{align}
\begin{aligned} \label{eq:p-32-2}
{\mathcal I}_{41} &= \int_{\mathbb{R}^4 \times \bbr_+^2} \gamma_o \bfv \cdot  \left(\bfu -\bfv \right)\left| \bfu -\bfv \right| F(\bfz)  d \bfz. 
 \end{aligned}
\end{align}
\noindent $\bullet$~Case D.2 (Estimate of ${\mathcal I}_{42}$):~By direct calculation, we obtain
\begin{align}
\begin{aligned} \label{eq:p-32-3}
{\mathcal I}_{42} & = \int_{\mathbb{R}^8 \times \bbr_+^4}   \gamma_{\bfn}  \chi(|\bfx^* - \bfx| - (r + r^*)) \bfv \cdot \bfn(\bfx, \bfx^*)  F( \bfz^*) F(\bfz) d\bfz^* d \bfz   \\
& = \int_{\mathbb{R}^8 \times \bbr_+^4}   \gamma_{\bfn}  \chi(|\bfx - \bfx^*| - (r^* + r)) \bfv^* \cdot \bfn(\bfx^*, \bfx)  F( \bfz) F(\bfz^*) d\bfz d \bfz^*   \\
&=  -\int_{\mathbb{R}^8 \times \bbr_+^4}   \gamma_{\bfn}  \chi(|\bfx^* - \bfx| - (r + r^*)) \bfv^* \cdot \bfn(\bfx, \bfx^*)  F( \bfz^*) F(\bfz) d\bfz^* d \bfz   \\
&= \frac{1}{2} \int_{\mathbb{R}^8 \times \bbr_+^4}   \gamma_{\bfn}  \chi(|\bfx^* - \bfx| - (r + r^*))( \bfv - \bfv^*)  \cdot \bfn(\bfx, \bfx^*)  F( \bfz^*) F(\bfz) d\bfz^* d \bfz, 
\end{aligned}
\end{align}
where we used a exchange map $\bfz~\longleftrightarrow~\bfz^*$ and $\bfn(\bfx^*, \bfx) = -\bfn(\bfx, \bfx^*)$. \newline

\noindent $\bullet$~Case D.3 (Estimate of ${\mathcal I}_{43}$): Similar to Case D.2, we have
\begin{align}
\begin{aligned} \label{eq:p-32-4}
{\mathcal I}_{43} & = \int_{\mathbb{R}^8 \times \bbr_+^4}   \gamma_{\bfv} [ (\bfv  - \bfv^*) \cdot \bfn(\bfx, \bfx^*)] \bfn(\bfx, \bfx^*)  \cdot \bfv F(\bfz) F(\bfz^*) d\bfz^* d\bfz \\
& = \int_{\mathbb{R}^8 \times \bbr_+^4}   \gamma_{\bfv} [ (\bfv^*  - \bfv) \cdot \bfn(\bfx^*, \bfx)] \bfn(\bfx^*, \bfx)  \cdot \bfv^* F(\bfz^*) F(\bfz) d\bfz d\bfz^* \\
& = -\int_{\mathbb{R}^8 \times \bbr_+^4}   \gamma_{\bfv} [ (\bfv  - \bfv^*) \cdot \bfn(\bfx, \bfx^*)] \bfn(\bfx, \bfx^*)  \cdot \bfv^* F(\bfz) F(\bfz^*) d\bfz^* d\bfz \\
& = \frac{1}{2}  \int_{\mathbb{R}^8 \times \bbr_+^4}   \gamma_{\bfv} [ (\bfv  - \bfv^*) \cdot \bfn(\bfx, \bfx^*)]^2  F(\bfz) F(\bfz^*) d\bfz^* d\bfz. \\
\end{aligned}
\end{align}
In \eqref{eq:p-32-1}, we combine \eqref{eq:p-32-2}, \eqref{eq:p-32-3} and \eqref{eq:p-32-4} to get the desired estimate in \eqref{eq:p-27}. 
\end{proof}

\begin{lemma} \label{lem:dEp}
Let $F = F(t, \bfz)$ be a global \rev{smooth} solution to \eqref{eq:p-24} which decays \rev{to zero} sufficiently fast at infinity in the phase space. Then, the following relation holds. 
\[
\frac{d\E_P(t)}{dt} =  -\frac{1}{2} \int_{\mathbb{R}^8 \times \bbr_+^4} \gamma_{\bfn}  \chi( (|\bfx^* - \bfx| - (r + r^*)) (\bfv - \bfv^*) \cdot \bfn(\bfx, \bfx^*) F(\bfz) F(\bfz^*)  d\bfz  d\bfz^*.
\]
\end{lemma}
\begin{proof} Note that \eqref{eq:p-24} implies 
\begin{align}
\begin{aligned} \label{eq:p-33-2}
&  \partial_t F(\bfz) +  \nabla_{\bfx} \cdot (\bfv F(\bfz) ) +   \nabla_{\bfv} \cdot (\bff[F](\bfz) F(\bfz))= 0, \\
&  \partial_t F(\bfz^*) +  \nabla_{\bfx^*} \cdot(\bfv^* F(\bfz^*)) +   \nabla_{\bfv^*} \cdot (\bff[F](\bfz^*) F(\bfz^*))= 0.
\end{aligned}
\end{align} 
Then, the relation $\eqref{eq:p-33-2}_1 F(\bfz^*) + F(\bfz) \eqref{eq:p-33-2}_2$  implies 
\begin{align}
\begin{aligned} \label{eq:p-33-3}
& \frac12 \partial_t \Big(F(\bfz) F(\bfz^*) \Big ) +  \nabla_{(\bfx, \bfx^*)} \cdot \Big(  (\bfv, \bfv^*) F(\bfz) F(\bfz^*) \Big) \\
& \hspace{2cm} +  \nabla_{(\bfv, \bfv^*)} \cdot \Big( (\bff[F](\bfz), \bff[F](\bfz^*)) F(\bfz) F(\bfz^*)        \Big) = 0.
\end{aligned}
\end{align}
We use \eqref{eq:p-33-3} and differentiate the second equation in \eqref{eq:p-25} (or \eqref{eq:Ep}) with respect to $t$ to find 
\begin{align}
\begin{aligned} \label{eq:p-33-4}
\frac{d \E_P}{dt} &= \frac{1}{4}  \int_{\bbr^8 \times \bbr_+^4} \gamma_{\bfn}  \Big( \int_0^{ (|\bfx^* - \bfx| - (r + r^*)} \chi(\eta) d\eta  \Big)     \partial_t   \Big( F(\bfz) F(\bfz^*) \Big) d\bfz  d\bfz^* \\
&= - \frac{1}{2} \int_{\mathbb{R}^8 \times \bbr_+^4} \gamma_{\bfn}  \Big( \int_0^{ (|\bfx^* - \bfx| - (r + r^*)} \chi(\eta) d\eta  \Big)  \nabla_{(\bfx, \bfx^*)} \cdot   \Big( (\bfv, \bfv^*) F(\bfz) F(\bfz^*) \Big) d\bfz  d\bfz^* \\
&- \frac{1}{2}  \int_{\mathbb{R}^8 \times \bbr_+^4} \gamma_{\bfn}  \Big( \int_0^{ (|\bfx^* - \bfx| - (r + r^*)} \chi(\eta) d\eta  \Big)  \nabla_{(\bfv, \bfv^*)} \cdot \Big( (\bff[F](\bfz), \bff[F](\bfz^*)) F(\bfz) F(\bfz^*)        \Big)  d\bfz  d\bfz^* \\
& =: {\mathcal I}_{51} + {\mathcal I}_{52}.
\end{aligned}
\end{align}
Next, we estimate the term ${\mathcal I}_{5i}$ one by one. \newline

\noindent $\bullet$~Case E.1 (Estimate on ${\mathcal I}_{51}$):~We use the integration by parts and the following relations similar to \eqref{eq:dddt}:
\begin{align*}
\begin{aligned}
&  \nabla_{\bfx} \Big( \int_0^{ (|\bfx^* - \bfx| - (r + r^*)} \chi(\eta) d\eta  \Big)   =  - \chi( (|\bfx^* - \bfx| - (r + r^*)) \bfn(\bfx^*, \bfx),  \\
&  \nabla_{\bfx^*}\Big( \int_0^{ (|\bfx^* - \bfx| - (r + r^*)} \chi(\eta) d\eta  \Big) = \chi( (|\bfx^* - \bfx| - (r + r^*)) \bfn(\bfx^*, \bfx)
\end{aligned}
\end{align*}
to find 
\begin{align}
\begin{aligned} \label{eq:p-33-5}
{\mathcal I}_{51} &= -\frac{1}{2} \int_{\mathbb{R}^8 \times \bbr_+^4}  \gamma_{\bfn} \Big( \int_0^{ (|\bfx^* - \bfx| - (r + r^*)} \chi(\eta) d\eta  \Big)  \nabla_{(\bfx, \bfx^*)} \cdot   \Big( (\bfv, \bfv^*) F(\bfz) F(\bfz^*) \Big) d\bfz  d\bfz^* \\
&= \frac{1}{2} \int_{\mathbb{R}^8 \times \bbr_+^4}  \gamma_{\bfn} \Big[ \nabla_{(\bfx, \bfx^*)} \Big( \int_0^{ (|\bfx^* - \bfx| - (r + r^*)} \chi(\eta) d\eta  \Big) \Big] \cdot  \Big( (\bfv, \bfv^*) F(\bfz) F(\bfz^*) \Big) d\bfz  d\bfz^* \\
& = \frac{1}{2}  \int_{\mathbb{R}^8 \times \bbr_+^4} \gamma_{\bfn}  \chi( (|\bfx^* - \bfx| - (r + r^*)) (\bfv^* - \bfv) \cdot \bfn(\bfx^*, \bfx) F(\bfz) F(\bfz^*)  d\bfz  d\bfz^*.
\end{aligned}
\end{align}
\noindent $\bullet$~Case E.2 (Estimate on ${\mathcal I}_{52}$):~By the divergence theorem, we have
\begin{align}
\begin{aligned} \label{eq:p-33-6}
{\mathcal I}_{52} &= -\frac{1}{2} \int_{\mathbb{R}^8 \times \bbr_+^4}  \gamma_{\bfn}  \nabla_{(\bfv, \bfv^*)} \cdot \Big[ \Big( \int_0^{ (|\bfx^* - \bfx| - (r + r^*)} \chi(\eta) d\eta  \Big)  \Big( (\bff[F](\bfz), \bff[F](\bfz^*)) F(\bfz) F(\bfz^*)        \Big) \Big]  d\bfz  d\bfz^*   \\
&= 0.
\end{aligned}
\end{align}
In \eqref{eq:p-33-4}, we combine \eqref{eq:p-33-5} and \eqref{eq:p-33-6} to find the desired estimate. 
\end{proof}
Finally, we combine Lemma \ref{lem:kmb} and Lemma \ref{lem:dEp} to derive the energy estimate. 
\begin{proposition} \label{thm:E}
Let $F = F(t, \bfz)$ be a global smooth solution to \eqref{eq:p-24} which decays \rev{to zero} sufficiently fast at infinity in phase space. Then, we have
\begin{align*}
\begin{aligned}
\frac{d\E}{dt} &=   \int_{\mathbb{R}^4 \times \bbr_+^2} \gamma_o \bfv \cdot  \left(\bfu -\bfv \right)\left| \bfu -\bfv \right| F(\bfz)  d \bfz  \\
&\hspace{0.5cm} + \frac{1}{2}  \int_{\mathbb{R}^8 \times \bbr_+^4}   \gamma_{\bfv} [ (\bfv  - \bfv^*) \cdot \bfn(\bfx, \bfx^*)]^2  F(\bfz) F(\bfz^*) d\bfz^* d\bfz.
\end{aligned}
\end{align*}
\end{proposition}

Next, we establish a kinetic counterpart of Theorem \ref{thm:v0}.  Suppose that the ocean velocity is constant, say zero:
\[  \bfu \equiv 0. \]
Let $\bfz = (\bfx, \bfv, r, h) \in \bbr^4 \times \bbr_+^2$ be fixed. Then, we define the particle trajectory $\bfz(t) = (\bfx(t), \bfv(t), r(t), h(t))$ issued from $\bfz$ at time $s = 0$ as the unique solution of the following system of integro-differential equations:
\begin{equation} \label{New-1}
\begin{cases}
\displaystyle \frac{d\bfx(t)}{dt} = \bfv(t), \quad t > 0, \\ 
\displaystyle  \frac{d\bfv(t)}{dt} = -\gamma_o \bfv(t) \left|\bfv(t) \right|  + \int_{\bbr^{4} \times \bbr_+^{2}}  \gamma_{\bfn}  \chi(|\bfx^* - \bfx(t)| - (r(t) + r^*)) \bfn(\bfx(t), \bfx^*)  F(t, \bfz^*) d\bfz^*  \\
\displaystyle \hspace{1.2cm} + \int_{\bbr^{4} \times \bbr_+^{2}}  \gamma_{\bfv}  [ (\bfv(t)  - \bfv^*) \cdot \bfn(\bfx(t), \bfx^*)] \bfn(\bfx(t), \bfx^*) F(t, \bfz^*) d\bfz^*,  \\
\displaystyle  \frac{dr(t)}{dt} = 0, \quad \frac{dh(t)}{dt} = 0,
\end{cases}
\end{equation} 
subject to initial data:
\begin{equation*} 
(\bfx, \bfv, r, h)(0) =  (\bfx, \bfv, r, h).
\end{equation*}
Then, it is easy to see that 
\[ r(t) = r, \quad h(t) = h, \quad t > 0. \]
Hence, the above system \eqref{New-1} can be rewritten as 
\begin{equation} \label{eq:p-34}
\begin{cases}
\displaystyle \frac{d\bfx(t)}{dt} = \bfv(t), \quad t > 0,. \\ 
\displaystyle  \frac{d\bfv(t)}{dt} = -\gamma_o \bfv(t) \left|\bfv(t) \right|  + \int_{\bbr^{4} \times \bbr_+^{2}}  \gamma_{\bfn}  \chi(|\bfx^* - \bfx(t)| - (r + r^*)) \bfn(\bfx(t), \bfx^*)  F(t, \bfz^*) d\bfz^*  \\
\displaystyle \hspace{1cm} + \int_{\bbr^{4} \times \bbr_+^{2}}  \gamma_{\bfv}  [ (\bfv(t)  - \bfv^*) \cdot \bfn(\bfx(t), \bfx^*)] \bfn(\bfx(t), \bfx^*) F(t, \bfz^*) d\bfz^*.
\end{cases}
\end{equation} 
Next, we establish the zero convergence of kinetic energy under the compact support assumption. First, we set $B_r(z)$ to be the open ball with the center $z$ and radius $r$. 
\begin{theorem} \label{thm:E0} 
Suppose that the ocean velocity, initial data, and system parameters satisfy
\[  \bfu = 0, \quad \E(0) < \infty, \quad \int_{\bbr^{4}\times \bbr_+^2} F^0(\bfz) d\bfz = 1,  \]
and let  $F = F(t,\bfz)$ be a global smooth solution to \eqref{eq:p-24} with a priori compact support assumption:
\begin{equation*} 
\exists~z_\infty < \infty\quad\mbox{such that}\quad\mbox{supp}_{\bfz}(F(t)) \subset B_{z_\infty}(0), \quad t > 0.
\end{equation*}
 Then, the kinetic energy tends to zero asymptotically:
\[ \lim_{t \to \infty}  \E_K(t) = 0.  \]
\end{theorem}
\begin{proof} With $\bfu = 0$ and \eqref{eq:p-34} in mind, Proposition \ref{thm:E} reduces to 
\begin{align}
\begin{aligned} \label{eq:p-34-1}
\frac{d\E}{dt} &=  -\int_{\mathbb{R}^4 \times \bbr_+^2} \gamma_o | \bfv |^3  F(\bfz)  d \bfz - \frac{1}{2}  \int_{\mathbb{R}^8 \times \bbr_+^4}   |\gamma_{\bfv}| [ (\bfv  - \bfv^*) \cdot \bfn(\bfx, \bfx^*)]^2  F(\bfz) F(\bfz^*) d\bfz^* d\bfz.
\end{aligned}
\end{align}  
Next, we define energy production terms: 
\begin{align*}
\begin{aligned}
\Lambda(t) &:= \Lambda_1(t) + \Lambda_2(t),  \quad \Lambda_1(t) := \int_{\mathbb{R}^4 \times \bbr_+^2} \gamma_o | \bfv |^3  F(t,\bfz)  d \bfz,        \\
\Lambda_2(t) &:= \frac{1}{2}  \int_{\mathbb{R}^8 \times \bbr_+^4}  | \gamma_{\bfv} | [ (\bfv  - \bfv^*) \cdot \bfn(\bfx, \bfx^*)]^2  F(t,\bfz) F(t,\bfz^*) d\bfz^* d\bfz.
\end{aligned}
\end{align*}
Then, we integrate the equation \eqref{eq:p-34-1} to find 
\begin{equation*} \label{eq:p-33-8}
\E(t) + \int_0^t \Lambda(s) ds = \E(0), \quad t \geq 0.
\end{equation*}
This yields
\begin{equation} \label{eq:p-34-2}
\sup_{0 \leq t < \infty} \E_K(t) \leq \E(t) \leq \E(0)
\end{equation}
and 
\begin{equation} \label{eq:p-34-3}
\int_0^{\infty} \Lambda_1(t) dt = \int_0^{\infty} \int_{\mathbb{R}^4 \times \bbr_+^2} \gamma_o  | \bfv |^3  F(t,\bfz)  d \bfz dt \leq \int_0^{\infty} \Lambda(t) dt \leq \E(0) < \infty.
\end{equation}
We claim that 
\begin{equation} \label{eq:p-34-4}
\sup_{0< t < \infty} \Big| \frac{d}{dt} \Lambda_1(t)  \Big| < \infty, \qquad \lim_{t \to \infty} \int_{\mathbb{R}^4 \times \bbr_+^2} | \bfv |^3  F(t,\bfz)  d \bfz = 0.
\end{equation}
{\it Proof of \eqref{eq:p-34-4}}: (i)~We use \eqref{eq:snot}, \eqref{eq:p-24} and zero far-field conditions of $F$ in $\bfz$-variable to find 
\begin{align}
\begin{aligned} \label{eq:p-34-5}
\frac{d}{dt} \Lambda_1(t) &= \int_{\mathbb{R}^4 \times \bbr_+^2} | \bfv |^3 \partial_t F(t,\bfz)  d \bfz \\
&= - \int_{\mathbb{R}^4 \times \bbr_+^2}   \nabla_{\bfx} \cdot  \Big( | \bfv |^3 \bfv F(t, \bfz) \Big)  d \bfz -\int_{\mathbb{R}^4 \times \bbr_+^2}   | \bfv |^3 \nabla_{\bfv} \cdot (\bff(F) F)  d \bfz \\
&= -\int_{\mathbb{R}^4 \times \bbr_+^2}   | \bfv |^3 \nabla_{\bfv} \cdot (\bff(F) F)  d \bfz =  3 \int_{\mathbb{R}^4 \times \bbr_+^2}  | \bfv | (\bfv \cdot \bff(F)) F  d \bfz.
\end{aligned}
\end{align}
On the other hand, there exists a positive constant $C$  such that 
\begin{align}
\begin{aligned} \label{eq:p-34-6}
& | \bff_{c,\bfn}[F](t,\bfz) |  \leq  C z_\infty, \quad |\bff_{c,\bfv}[F](t,\bfz)| \leq C z_\infty, \\
&  | \bfv|  | \bfv \cdot \bff(F)| = |\bfv| \cdot   \Big |\bfv  \cdot \Big (   -\gamma_o \bfv  |\bfv | +   \bff_{c,\bfn}[F] +  \bff_{c,\bfv}[F]  \Big ) \Big| \leq C ( z_\infty^4 + z_\infty^3).
\end{aligned}
\end{align}
We combine \eqref{eq:p-34-5}, \eqref{eq:p-34-6}  and use the unit mass of $F$ to see the desired first estimate in \eqref{eq:p-34-4}:
\[
\Big| \frac{d}{dt} \Lambda_1(t)  \Big| \leq 3C ( z_\infty^4 + z_\infty^3)  \int_{\mathbb{R}^4 \times \bbr_+^2} F  d \bfz \leq  3C ( z_\infty^4 + z_\infty^3).
\]
Since the right-hand side of the above inequality is independent of $t$, we take the supremum of the above relation over $t$ to find the desired first estimate.  \newline

\noindent (ii)~Note that the first estimate  in \eqref{eq:p-34-4} implies the uniform continuity of $\Lambda_1$. Hence, we use \eqref{eq:p-34-2}, \eqref{eq:p-34-3} and \rev{Barbalat's} lemma to find 
\begin{equation} \label{eq:p-34-7}
\lim_{t \to \infty} \int_{\mathbb{R}^4 \times \bbr_+^2} | \bfv |^3  F(t,\bfz)  d \bfz = 0.
\end{equation}
Now, we use the H\"{o}lder inequality to get 
\begin{align}
\begin{aligned} \label{eq:p-34-8}
\E_K(t) &= \frac{1}{2} \int_{\mathbb{R}^4 \times \bbr_+^2} |\bfv|^2 F(\bfz) \d \bfz \leq \frac{1}{2} \int_{\mathbb{R}^4 \times \bbr_+^2} \Big( |\bfv|^2 F^{\frac{2}{3}} (\bfz) \Big) F^{\frac{1}{3}}(\bfz)  \d \bfz \\
&\leq  \frac{1}{2}  \Big( \int_{\mathbb{R}^4 \times \bbr_+^2}  |\bfv|^3 F(\bfz) \d \bfz \Big)^{\frac{2}{3}}  \Big(\int_{\mathbb{R}^4 \times \bbr_+^2}  F(\bfz) \d \bfz \Big)^{\frac{1}{3}}  \\
& =  \frac{1}{2}  \Big( \int_{\mathbb{R}^4 \times \bbr_+^2}  |\bfv|^3 F(\bfz) \d \bfz \Big)^{\frac{2}{3}}.
\end{aligned}
\end{align}
Finally, we use \eqref{eq:p-34-7} and \eqref{eq:p-34-8} to find the desired estimate: 
\begin{equation*}  
\lim_{t \to \infty} \E_K(t) = 0.
\end{equation*}
\end{proof}

\rev{Herein, we used the a priori compact support assumption. While we do not have proof that this compact assumption is valid for all times, such assumptions are commonly adopted in kinetic theory to ensure well-posedness and boundedness of moments (see, e.g., \cite{spohn2012large}).}

\begin{remark} 
By the Cauchy-Schwarz inequality, one has 
\[ \int_{\bbr^{4} \times \bbr_+^2}  |\bfv | F(\bfz) d\bfz \leq \sqrt{ \int_{\bbr^{4} \times \bbr_+^2}  \| \bfv \|^2  F(\bfz)  d \bfz} = \sqrt{2\E_K(t)}.\]
Then, we use Theorem \ref{thm:E0} and the above relation to find 
\begin{equation*} 
\lim_{t \to \infty}  \int_{\bbr^{4} \times \bbr_+^2}  |\bfv | F(\bfz) d\bfz = 0.
\end{equation*}
\end{remark}

\section{From kinetic to hydrodynamic description} \label{sec:hmodel}
In this section, we discuss the hydrodynamic description of the collective behaviors of many ice floes. For this, we derive hydrodynamic observables from the kinetic model \eqref{eq:p-24} under a suitable closure condition:
\begin{equation} \label{eq:p-37}
 \partial_t F +   \nabla_{\bfx} \cdot (\bfv F) +   \nabla_{\bfv} \cdot (\bff[F] F)= 0.
\end{equation}
For notational simplicity, we simply set 
\[  \bfz := (\bfv, r, h) \in \bbr^2 \times \bbr_+^2, \quad  d \bfz := d\bfv dr dh. \]

\subsection{Hydrodynamic model for ice floe dynamics}

We first introduce hydrodynamic observables corresponding to the velocity-radius-thickness averages of the one-particle distribution function $F = \rev{F(t, \bfx, \bfz)}$:
\begin{equation}
\begin{cases} \label{eq:p-38}
\displaystyle \rho(t, \bfx) :=  \int_{\bbr^2 \times \bbr_+^2} F d \bfz, \quad &\mbox{local mass density},   \\
\displaystyle  (\rho \bfuu)(t, \bfx) :=  \int_{\bbr^2 \times \bbr_+^2} \bfv F  \d \bfz ,  \quad &\mbox{local momentum density}, \\
\displaystyle (\rho E)(t, \bfx) :=   \int_{\bbr^2 \times \bbr_+^2} \frac{1}{2} | \bfv |^2 F  \d \bfz ,  \quad &\mbox{local energy density}, \\
\displaystyle (\rho e)(t, \bfx) :=   \int_{\bbr^2 \times \bbr_+^2} \frac{1}{2} | \bfv - \bfuu |^2 F  \d \bfz ,  \quad &\mbox{local internal energy density}, \\
\displaystyle ({\rho E}_K)(t, \bfx) := \frac{1}{2}( \rho |\bfuu|^2)(t, x), \quad  &\mbox{local kinetic energy density}.
\end{cases}
\end{equation}
Then, we use 
\[ \frac{1}{2} | \bfv |^2 = \frac{1}{2} | \bfv - \bfuu + \bfuu |^2 =  \frac{1}{2} | \bfv - \bfuu |^2 + \frac{1}{2} |\bfuu|^2 +  (\bfv - \bfuu) \cdot \bfuu \]
to see the following relation:
\begin{equation*}
\begin{aligned} 
\rho E &=  \frac{1}{2} \int_{\bbr^2 \times \bbr_+^2} | \bfv |^2 F \d \bfz \\
& =\frac{1}{2}  \int_{\bbr^2 \times \bbr_+^2} | \bfv - \bfuu |^2 F \d \bfz  + \frac{1}{2} \int_{\bbr^2 \times \bbr_+^2}  |\bfuu|^2 F  \d \bfz  +   \bfuu \cdot \int_{\bbr^2 \times \bbr_+^2}  (\bfv - \bfuu)  F  \d \bfz  \\
&=: \rho e +  \rho E_K.
\end{aligned}
\end{equation*}
We also define the stress tensor $P = (p_{ij})$ and heat flux $q = (q_i)$ as follows:
\begin{align}
\begin{aligned} \label{eq:p-39}
p_{ij} &:= \int_{\bbr^2 \times \bbr_+^2}(v_i - u_i)(v_j - u_j) F  \d \bfz, \quad \forall~i, j  = 1,2, \\
q_i &:=  \int_{\bbr^2 \times \bbr_+^2} (v_i - u_i) \| \bfv - \bfuu \|^2 F  \d \bfz.
\end{aligned}
\end{align}

In the following lemma, we derive the local balanced laws for hydrodynamic observables $(\rho, \rho \bfuu, \rho E)$. 
\begin{lemma}[Hydrodynamic description] \label{lem:conteq}
Let $F = F(t, \bfz)$ be a global smooth solution to \eqref{eq:p-37} which decays \rev{to zero} sufficiently fast at infinity in $\bfv, r$ and $h$. Then, the hydrodynamic observables $(\rho, \rho \bfuu,  \rho E)$ satisfy local balanced laws:
\begin{equation} \label{eq:hmodel}
\begin{cases}
\displaystyle \partial_t \rho + \nabla_{\bfx} \cdot (\rho \bfuu) = 0, \\
\displaystyle  \partial_t (\rho \bfuu)+  \nabla_{\bfx} \cdot \Big ( \rho \bfuu \otimes \bfuu + P \Big) =   \int_{\bbr^2 \times \bbr_+^2} \rev{\gamma_o}
    \left(\bfu -\bfv\right)\left| \bfu -\bfv\right| F(\bfz)   \d \bfz,   \\
 \displaystyle  \partial_t (\rho E)+  \nabla_{\bfx} \cdot \Big ( \rho E \bfuu + Pu + q \Big) =  \int_{\mathbb{R}^2 \times \bbr_+^2} \gamma_o \bfv \cdot  \left(\bfu -\bfv \right)\left| \bfu -\bfv \right| F(\bfz)   \d \bfz   \\
 \displaystyle  \hspace{0.5cm}+ \frac{1}{2} \int_{\mathbb{R}^4 \times \bbr_+^4}   \gamma_{\bfn}  \chi(|\bfx^* - \bfx| - (r + r^*))( \bfv - \bfv^*)  \cdot \bfn(\bfx, \bfx^*)  F( \bfz^*) F(\bfz) \d \bfz^*   \d \bfz   \\
  \displaystyle  \hspace{0.5cm} +\frac{1}{2}  \int_{\mathbb{R}^4 \times \bbr_+^4}   \gamma_{\bfv}  \big|  (\bfv  - \bfv^*) \cdot \bfn(\bfx, \bfx^*) \big |^2  F(\bfz) F(\bfz^*) \d\bfz^*  \d \bfz.  \\
\end{cases}
\end{equation}
\end{lemma}
\begin{proof} The local balanced laws will be derived using the temporal evolution of the velocity moments up to the second order.  \newline

\noindent (i)~(Conservation law for mass): We integrate \eqref{eq:p-37} with respect to $\bfz = (\bfv, r, h)$ using far-field decay conditions of $F$ to find 
\[ 
 \int_{\bbr^2 \times \bbr_+^2} \Big(  \partial_t F +  \nabla_{\bfx} \cdot (\bfv F) +   \nabla_{\bfv} \cdot (\bff[F] F)  \Big)  \d \bfz = 0.
\]
Using the definition \eqref{eq:p-38}, this yields
\[
\partial_t  \Big( \int_{\bbr^2 \times \bbr_+^2}  F(\bfz)   \d \bfz  \Big) + \nabla_{\bfx} \cdot \Big(  \int_{\bbr^2 \times \bbr_+^2}  \bfv  F(\bfz)   \d \bfz  \Big) = 0,
\]
i.e., we have the continuity equation:
\[
\partial_t \rho + \nabla_{\bfx} \cdot (\rho \bfuu) = 0.
\]

\vspace{0.2cm}

\noindent (ii)~(Balance law for momentum):~We take an inner product of \eqref{eq:p-37} with $\bfv$ to find 
\begin{equation} \label{eq:p-40}
 \partial_t (\bfv F)  +  \nabla_{\bfx} \cdot (\bfv \otimes \bfv F) +  \nabla_{\bfv} \cdot ( \bfv \otimes \bff[F] F) = \bff[F] F.
\end{equation}
We integrate \eqref{eq:p-40} over $\bbr^2 \times \bbr_+^2$ to get 
\begin{align}
\begin{aligned} \label{eq:p-41}
& \partial_t (\rho \bfuu)+  \nabla_{\bfx} \cdot \Big ( \int_{\bbr^2 \times \bbr_+^2} \bfv \otimes \bfv F  \d \bfz  \Big)  \\
& \hspace{1cm} = \int_{\bbr^2 \times \bbr_+^2} \bff_o F   \d \bfz  +  \int_{\bbr^2 \times \bbr_+^2}   \bff_{c,\bfn}[F] F  \d \bfz  + 
 \int_{\bbr^2 \times \bbr_+^2}  \bff_{c,\bfv}[F] F  \d \bfz  \\
 & \hspace{1cm} =: {\mathcal I}_{61}+ {\mathcal I}_{62}+ {\mathcal I}_{63}.
\end{aligned}
\end{align}
Below, we estimate the terms $\displaystyle \int_{\bbr^2 \times \bbr_+^2} \bfv \otimes \bfv F  d\bfz$ and ${\mathcal I}_{6i}$ one by one. \newline

\noindent $\bullet$~Case F.1: Note that 
\begin{equation} \label{eq:p-42}
\int_{\bbr^2 \times \bbr_+^2} \bfv \otimes  \bfv F  \d \bfz  =   \rho \bfuu \otimes \bfuu + P. 
\end{equation}

\vspace{0.2cm}

\noindent $\bullet$~Case F.2: By direct calculation, one has 
\begin{equation} \label{eq:p-43}
 {\mathcal I}_{61} = \int_{\bbr^2 \times \bbr_+^2}  \rev{\gamma_o}
    \left(\bfu -\bfv\right)\left| \bfu -\bfv\right| F(t,\bfz)  \d \bfz.
\end{equation}

\vspace{0.2cm}

\noindent $\bullet$~Case F.3: By the same analysis as in Case C.1 in the proof of Lemma \ref{lem:kmb}, we have
\begin{align}
\begin{aligned} \label{eq:p-44}
{\mathcal I}_{62} &=   \int_{\bbr^2 \times \bbr_+^2}   \bff_{c,\bfn}[F] F   \d \bfz   \\
&=   \int_{\bbr^{4} \times \bbr_+^{4}}  \gamma_{\bfn}  \chi(|\bfx^* - \bfx| - (r + r^*)) \bfn(\bfx, \bfx^*)  F(\bfz^*) F(\bfz) \d \bfz^*  \d \bfz \\
& = 0.
\end{aligned}
\end{align}
Here, we use the index exchange $\bfz~\leftrightarrow~\bfz^*$.   

\vspace{0.2cm}

\noindent $\bullet$~Case F.4:~Similar to Case F.3, we have
\begin{align}
\begin{aligned}  \label{eq:p-45}
 {\mathcal I}_{63} &=  \int_{\bbr^{4} \times \bbr_+^{4}}   \gamma_{\bfv} [ (\bfv  - \bfv^*) \cdot \bfn(\bfx, \bfx^*)] \bfn(\bfx, \bfx^*)  F(\bfz^*) F(\bfz) \d\bfz^*  \d \bfz  \\
&=  -\int_{\bbr^{4} \times \bbr_+^{4}}   \gamma_{\bfv} [ (\bfv^*  - \bfv) \cdot \bfn(\bfx, \bfx^*)] \bfn(\bfx, \bfx^*) F(\bfz^*) F(\bfz) \d\bfz^* \d\bfz \\
& = 0.
\end{aligned}
\end{align}
In \eqref{eq:p-41}, we combine in \eqref{eq:p-42}, \eqref{eq:p-43}, \eqref{eq:p-44} and \eqref{eq:p-45} to get the desired balance law for momentum. 

\vspace{0.2cm}

\noindent (iii)~(Balance law for energy): It follows from \eqref{eq:p-32} with integration over $\bbr^2 \times \bbr_+^2$ that 
\begin{align}
\begin{aligned} \label{eq:p-47}
& \partial_t (\rho E_K) +  \nabla_{\bfx} \cdot \Big( \rho E_K \bfuu + P \bfuu + q   \Big) \\
& \hspace{0.5cm} = \int_{\mathbb{R}^2 \times \bbr_+^2}  (\bfv \cdot \bff_o) F(\bfz)  \d \bfz  + \int_{\mathbb{R}^2 \times \bbr_+^2}   (\bfv \cdot \bff_{c,\bfn}[F])(\bfz) F(\bfz)  \d \bfz  
\\
& \hspace{0.5cm} + \int_{\mathbb{R}^2 \times \bbr_+^2} (\bfv \cdot  \bff_{c,\bfv}[F])(\bfz) F(\bfz)  \d \bfz \\
& \hspace{0.5cm} =:  {\mathcal I}_{71} + {\mathcal I}_{72} + {\mathcal I}_{73}.
\end{aligned}
\end{align}
Next, we estimate the terms one by one.  \newline

\noindent $\bullet$~Case G.1: By direct calculation, one has 
\begin{equation} \label{eq:p-48}
{\mathcal I}_{71}  = \int_{\mathbb{R}^2 \times \bbr_+^2}  (\bfv \cdot \bff_o) F(\bfz)  \d \bfz  =   \int_{\mathbb{R}^2 \times \bbr_+^2} \gamma_o \bfv \cdot  \left(\bfu -\bfv \right)\left| \bfu -\bfv \right| F(\bfz)   \d \bfz .
 \end{equation}

\vspace{0.2cm}

\noindent $\bullet$~Case G.2: By direct calculation, we have
\begin{align}
\begin{aligned}  \label{eq:p-49}
{\mathcal I}_{72} & = \int_{\mathbb{R}^4 \times \bbr_+^4}   \gamma_{\bfn}  \chi(|\bfx^* - \bfx| - (r + r^*)) \bfv \cdot \bfn(\bfx, \bfx^*)  F( \bfz^*) F(\bfz) \d\bfz^*  \d \bfz   \\
&=  -\int_{\mathbb{R}^4 \times \bbr_+^4}   \gamma_{\bfn}  \chi(|\bfx^* - \bfx| - (r + r^*)) \bfv^* \cdot \bfn(\bfx, \bfx^*)  F( \bfz^*) F(\bfz) \d\bfz^*  \d \bfz    \\
&= \frac{1}{2} \int_{\mathbb{R}^4 \times \bbr_+^4}   \gamma_{\bfn}  \chi(|\bfx^* - \bfx| - (r + r^*))( \bfv - \bfv^*)  \cdot \bfn(\bfx, \bfx^*)  F( \bfz^*) F(\bfz) \d \bfz^*  \d \bfz.
\end{aligned}
\end{align}

\vspace{0.2cm}

\noindent $\bullet$~Case G.3: Similarly, we have
\begin{align}
\begin{aligned} \label{eq:p-50}
{\mathcal I}_{73} & = \int_{\mathbb{R}^4 \times \bbr_+^4}   \gamma_{\bfv} [ (\bfv  - \bfv^*) \cdot \bfn(\bfx, \bfx^*)]  \bfv \cdot  \bfn(\bfx, \bfx^*) F(\bfz) F(\bfz^*) \d\bfz^* \d \bfz  \\
&= -\int_{\mathbb{R}^4 \times \bbr_+^4}   \gamma_{\bfv} [ (\bfv  - \bfv^*) \cdot \bfn(\bfx, \bfx^*)] \bfv^* \cdot  \bfn(\bfx, \bfx^*) F(\bfz) F(\bfz^*) \d\bfz^*  \d \bfz  \\
& = \frac{1}{2}  \int_{\mathbb{R}^4 \times \bbr_+^4}  \gamma_{\bfv}  \big|  (\bfv  - \bfv^*) \cdot \bfn(\bfx(t), \bfx^*) \big |^2   F(\bfz) F(\bfz^*) \d\bfz^*  \d \bfz.
\end{aligned}
\end{align}
In \eqref{eq:p-47}, we combine \eqref{eq:p-48}, \eqref{eq:p-49} and \eqref{eq:p-50} to find the desired balance law for energy. 
\end{proof}

\rev{To close the system, an appropriate ansatz or assumption for hydrodynamic equilibrium is needed. In the literature of classical kinetic equations, several closure conditions have been used, for example, the local Maxwellian assumption and the maximum entropy principle for the derivation of conservation laws from Boltzmann equation \cite{bardos1991fluid,saint2009hydrodynamic,golse2005boltzmann}. For Vlasov equation describing flocking dynamics, one may approximate the kinetic equation in a close-to-flocking regime in which the velocity configuration is close to velocity alignment. Thus, the mono-kinetic ansatz is natural to take as a closure condition. For the kinetic Cucker-Smale model, this mono-kinetic ansatz was rigorously justified in \cite{figalli2018rigorous}.  Following \cite{figalli2018rigorous}, we} assume the mono-kinetic ansatz for $F$:
\begin{equation} \label{eq:rhoPG}
F(t, \bfx, \bfv,r, h) = \rho(t, \bfx) \delta(\bfv - \bfuu(t, \bfx)) \otimes P(r) \otimes \Gamma(h). 
\end{equation}
where $\delta$ is the two-dimensional Dirac delta function, $P(r)$ is the floe size distribution in \eqref{floe_size_pdf}, and $\Gamma(h)$ is the Gamma distribution in \eqref{floe_thickness_pdf} for thickness\rev{; see \ref{app:a} for details}. 
In this case,  we apply the \rev{property of Dirac delta function} and substitute the ansatz \eqref{eq:rhoPG} into  \eqref{eq:p-39} to obtain
\begin{align}
\begin{aligned} \label{eq:rhoPG-1}
p_{ij} &=   \int_{\bbr^2 \times \bbr_+^2}(v_i - u_i)(v_j - u_j) F  \d \bfz  \\
&=  \int_{\bbr^2 \times \bbr_+^2}(v_i - u_i)(v_j - u_j) \rho(t, \bfx) \delta(\bfv - \bfuu(t, \bfx)) \otimes P(r) \otimes \Gamma(h) \d \bfz  \\
&= 0, \quad i, j \in [2], \\
 q_i &= \int_{\bbr^2 \times \bbr_+^2} (v_i - u_i)  | \bfv - \bfuu |^2 F  \d \bfz  \\
 &=  \int_{\bbr^2 \times \bbr_+^2} (v_i - u_i)  | \bfv - \bfuu |^2  \delta(\bfv - \bfuu(t, \bfx)) \otimes P(r) \otimes \Gamma(h)  \d \bfz  \\
 &=0, \quad i \in [2].
 \end{aligned}
 \end{align}

 In the sequel, we simplify the balanced laws \eqref{eq:hmodel} using the mono-kinetic ansatz \eqref{eq:rhoPG}, \eqref{eq:rhoPG-1}, and the unit mass condition. In this case, \rev{using \eqref{eq:snot}}, the non-local sources can be rewritten as follows.
 \begin{align}
\begin{aligned} \label{eq:p-51}
& \bullet~P = 0, \quad q = 0, \\
& \bullet~ \int_{\bbr^2 \times \bbr_+^2} \rev{\gamma_o} \left(\bfu -\bfv\right)\left| \bfu -\bfv\right| F(\bfz)  \d \bfz \\
& \hspace{0.7cm} = \int_{\bbr^2 \times \bbr_+^2} \rev{\gamma_o} \left(\bfu -\bfv\right)\left| \bfu -\bfv\right|  \rho(t, \bfx) \delta_D(\bfv - \bfuu(t, \bfx)) \otimes P(r) \otimes \Gamma(h)  \d \bfz  \\
& \hspace{0.7cm} = \rev{\alpha} (\bfu -\bfuu)\left| \bfu -\bfuu \right|, \\
&  \bullet~\int_{\mathbb{R}^2 \times \bbr_+^2} \rev{\gamma_o} \bfv \cdot  \left(\bfu -\bfv \right)\left| \bfu -\bfv \right| F(\bfz)   \d \bfz  \\
& \hspace{0.7cm} =  \int_{\mathbb{R}^2 \times \bbr_+^2} \rev{\gamma_o} \bfv \cdot  \left(\bfu -\bfv \right)\left| \bfu -\bfv \right| \rho(t, \bfx) \delta_D(\bfv - \bfuu(t, \bfx)) \otimes P(r) \otimes \Gamma(h)     \d \bfz  \\
& \hspace{0.7cm} = \rev{\alpha} \bfuu \cdot  \left(\bfu -\bfuu \right)\left| \bfu -\bfuu \right|, 
\end{aligned}
\end{align}
and
\begin{align}
\begin{aligned} \label{eq:p-51-1}
& \bullet~ \int_{\bbr^{4} \times \bbr_+^{4}}  \gamma_{\bfn}  \chi(|\bfx^* - \bfx| - (r + r^*)) \bfn(\bfx, \bfx^*)  F(\bfz^*) F(\bfz) \d \bfz^* \d \bfz   \\
& \hspace{0.7cm} =  \int_{\bbr^{4} \times \bbr_+^{4}}  \gamma_{\bfn}  \chi(|\bfx^* - \bfx| - (r + r^*)) \bfn(\bfx, \bfx^*) \rho(t, \bfx^*) \delta_D(\bfv - \bfuu(t, \bfx^*)) \otimes P(r^*) \otimes \Gamma(h^*)    \\
& \hspace{1cm}  \times \rho(t, \bfx) \delta_D(\bfv - \bfuu(t, \bfx)) \otimes P(r) \otimes \Gamma(h)   \d \bfz^* \d \bfz  \\
& \hspace{0.7cm} =\rho \int_{\partial B(\bfx, 2r^{\infty})}  \rev{\kappa_1} \chi(|\bfx^* - \bfx| - (r + r^*))  \bfn(\bfx, \bfx^*)  \rho (t,\bfx^*)  d\sigma_{\bfx^*}, \\
& \bullet~ \frac{1}{2} \int_{\mathbb{R}^4 \times \bbr_+^4}   \gamma_{\bfn}  \chi(|\bfx^* - \bfx| - (r + r^*))( \bfv - \bfv^*)  \cdot \bfn(\bfx, \bfx^*)  F( \bfz^*) F(\bfz) \d\bfz^*   \d \bfz  \\
& \hspace{1cm} = \rev{\frac{1}{2}}\int_{\partial B(\bfx, 2r^{\infty})} \rev{\kappa_1} \chi(|\bfx^* - \bfx| - (r + r^*))  ( \bfuu(\bfx) - \bfuu(\bfx^*))  \cdot \bfn(\bfx, \bfx^*)  \rho(t,\bfx^*)  \d\sigma_{\bfx^*},  \\
&  \bullet~\frac{1}{2}  \int_{\mathbb{R}^4  \times \bbr_+^4}    \gamma_{\bfv}  \big|  (\bfv  - \bfv^*) \cdot \bfn(\bfx, \bfx^*) \big |^2   F(\bfz) F(\bfz^*) \d\bfz^*  \d \bfz  \\
& \hspace{1cm} = \rev{\frac{1}{2}} \int_{\partial B(\bfx, 2r^{\infty})}   \rev{\kappa_2}  \big|  (\bfuu  - \bfuu^*) \cdot \bfn(\bfx, \bfx^*) \big |^2  \rho(t, \bfx^*) \d\sigma_{\bfx^*},
\end{aligned}
\end{align}
where $\partial B_{2r^{\infty}}(\bfx)$ is the ball with a center $\bfx$ and a radius $2r^{\infty}$, and $d\sigma_{\bfx^*}$ is the surface measure on the circle $\partial B(\bfx, 2r^{\infty})$.  \newline

Finally, we combine \eqref{eq:hmodel},  \eqref{eq:p-51}, and \eqref{eq:p-51-1} to find 
\begin{equation} \label{eq:hpgmodel}
\begin{cases}
\displaystyle \partial_t \rho + \nabla_{\bfx} \cdot (\rho \bfuu) = 0, \\
\displaystyle  \partial_t (\rho \bfuu)+  \nabla_{\bfx} \cdot \Big ( \rho \bfuu \otimes \bfuu  \Big) = \rev{\alpha}(\bfu -\bfuu)\left| \bfu -\bfuu \right |,  
\\
\displaystyle  \partial_t (\rho E)+  \nabla_{\bfx} \cdot \Big ( \rho \bfuu E \Big) =  \rev{\alpha} \bfuu \cdot  \left(\bfu -\bfuu \right)\left| \bfu -\bfuu \right|  \\
 \displaystyle  +   \rev{\frac{1}{2}}\int_{\partial B(\bfx, 2r^{\infty})} \rev{\kappa_1} \chi(|\bfx^* - \bfx| - (r + r^*))  ( \bfuu(\bfx) - \bfuu(\bfx^*))  \cdot \bfn(\bfx, \bfx^*)  \rho(t,\bfx^*)  \d\sigma_{\bfx^*} \\
  \displaystyle +  \rev{\frac{1}{2}} \int_{\partial B(\bfx, 2r^{\infty})}  \rev{\kappa_2}  \big|  (\bfuu  - \bfuu^*) \cdot \bfn(\bfx, \bfx^*) \big |^2  \rho(t, \bfx^*) \d\sigma_{\bfx^*}.
\end{cases}
\end{equation}
In particular, for the case $\bfu = 0$ in \eqref{eq:hpgmodel}, we have
\begin{equation} \label{eq:hpgmodel0}
\begin{cases}
\displaystyle \partial_t \rho + \nabla_{\bfx} \cdot (\rho \bfuu) = 0, \\
\displaystyle  \partial_t (\rho \bfuu)+  \nabla_{\bfx} \cdot \Big ( \rho \bfuu \otimes \bfuu  \Big)  = - \rev{\alpha} \bfuu \left| \bfuu \right |, \\
\displaystyle  \partial_t (\rho E)+  \nabla_{\bfx} \cdot \Big ( \rho  \bfuu E \Big) =   - \rev{\alpha} |\bfuu|^3  \\
 \displaystyle  +  \rev{\frac{1}{2}} \int_{\partial B(\bfx, 2r^{\infty})}   \rev{\kappa_1}   \chi(|\bfx^* - \bfx| - (r + r^*))  ( \bfuu(\bfx) - \bfuu(\bfx^*))  \cdot \bfn(\bfx, \bfx^*) \rho(t,\bfx^*)  \d\sigma_{\bfx^*}  \\
  \displaystyle + \rev{\frac{1}{2}} \int_{\partial B(\bfx, 2r^{\infty})}   \rev{\kappa_2}  \big|  (\bfuu  - \bfuu^*) \cdot \bfn(\bfx, \bfx^*) \big |^2  \rho(t, \bfx^*) \d\sigma_{\bfx^*}.
\end{cases}
\end{equation}
%
%

\subsection{Macroscopic behavior of the hydrodynamic model}
Similar to the case for the kinetic model, we define energy functionals:
\begin{align}
\begin{aligned} \label{eq:he}
\E &:= \E_K + \E_P,  \quad   \E_K := \frac{1}{2} \int_{\bbr^2} \rho(\bfx) |\bfuu(\bfx)|^2 d \bfx, \\
\E_P &:=  \frac{1}{4} \int_{\mathbb{R}^8  \times \bbr_+^4}  \gamma_{\bfn}  \chi(|\bfx^* - \bfx| - (r + r^*))^2  F(\bfz) F(\bfz^*) \d\bfz^*  \d \bfz \d\bfx^*  \d \bfx. 
\end{aligned}
\end{align}

With the continuum model developed in Lemma \ref{lem:conteq}, 
in the following lemma, we establish the temporal evolutions of $\E, \E_K$, and $\E_P$, respectively. 
\begin{lemma}[Energy behavior]
Let $(\rho, \bfuu)$ be a global smooth solution to the system \eqref{eq:hpgmodel0}. Then, we have
\begin{align*}
\begin{aligned}
& (i)~\frac{d\E_K}{dt} =  -  \int_{\bbr^2} \rho |\bfuu|^3 \d \bfx \\
&~~ +    \frac{1}{2}\int_{\bbr^2  \times  \partial B(\bfx, 2r^{\infty})}  \gamma_{\bfn}  \chi(|\bfx^* - \bfx| - (r + r^*))  ( \bfuu(\bfx) - \bfuu(\bfx^*))  \cdot \bfn(\bfx, \bfx^*)  \rho(\bfx) \rho(\bfx^*)  \d\sigma_{\bfx^*}  \d \bfx  \\
&~~+  \frac{1}{2}\int_{\bbr^2  \times  \partial B(\bfx, 2r^{\infty})}   \gamma_{\bfv}  \big|  (\bfuu  - \bfuu^*) \cdot \bfn(\bfx, \bfx^*) \big |^2  \rho(\bfx) \rho(\bfx^*) \d\sigma_{\bfx^*} \d \bfx, \\
& (ii)~ \frac{d \E_P}{dt} =  - \frac{1}{2}\int_{\bbr^2  \times  \partial B(\bfx, 2r^{\infty})}  \gamma_{\bfn}  \chi(|\bfx^* - \bfx| - (r + r^*))  ( \bfuu(\bfx) - \bfuu(\bfx^*))  \cdot \bfn(\bfx, \bfx^*)  \rho(\bfx) \rho(\bfx^*)  \d\sigma_{\bfx^*}  \d \bfx, \\
& (iii)~ \frac{d \E}{dt} =  -  \int_{\bbr^2} \rho |\bfuu|^3 \d \bfx +  \frac{1}{2}\int_{\bbr^2  \times  \partial B(\bfx, 2r^{\infty})}   \gamma_{\bfv}  \big|  (\bfuu  - \bfuu^*) \cdot \bfn(\bfx, \bfx^*) \big |^2  \rho(\bfx) \rho(\bfx^*) \d\sigma_{\bfx^*} \d \bfx.
\end{aligned}
\end{align*}
\end{lemma}
\begin{proof} 
(i)~We integrate $\eqref{eq:hpgmodel0}_3$ with respect to $\bfx$ over $\bbr^2$ to get 
\begin{align*}
\begin{aligned}
\frac{d\E_K}{dt} &= \frac{d}{dt} \int_{\bbr^2}   \rho E_K \d \bfx = \int_{\bbr^2}  \partial_t (\rho E_K)+  \nabla_{\bfx} \cdot \Big ( \rho E_K \bfuu \Big)   \d \bfx \\
&=  -  \int_{\bbr^2} \rho |\bfuu|^3 \d \bfx  \\
&~+    \frac{1}{2}\int_{\bbr^2  \times  \partial B(\bfx, 2r^{\infty})}  \gamma_{\bfn} \chi(|\bfx^* - \bfx| - (r + r^*))  ( \bfuu(\bfx) - \bfuu(\bfx^*))  \cdot \bfn(\bfx, \bfx^*)  \rho(\bfx) \rho(\bfx^*)  \d\sigma_{\bfx^*}  \d \bfx  \\
&~+  \frac{1}{2}\int_{\bbr^2  \times  \partial B(\bfx, 2r^{\infty})}   \gamma_{\bfv}  \big|  (\bfuu  - \bfuu^*) \cdot \bfn(\bfx, \bfx^*) \big |^2   \rho(\bfx) \rho(\bfx^*)  \d\sigma_{\bfx^*} \d \bfx.
\end{aligned}
\end{align*}
(ii)~Following \eqref{eq:dddt}, we derive 
\begin{align*}
\begin{aligned}
\frac{d\E_P}{dt} &= \frac{d}{dt} \frac{1}{4} \int_{\mathbb{R}^8  \times \bbr_+^4}  \gamma_{\bfn}  \chi(|\bfx^* - \bfx| - (r + r^*))^2  F(\bfz) F(\bfz^*) \d\bfz^*  \d \bfz \d\bfx^*  \d \bfx \\
&= \frac{1}{4} \int_{\mathbb{R}^8  \times \bbr_+^4}  \gamma_{\bfn}  2\chi(|\bfx^* - \bfx| - (r + r^*)) \frac{d \chi(|\bfx^* - \bfx| - (r + r^*))}{dt}  F(\bfz) F(\bfz^*) \d\bfz^*  \d \bfz \d\bfx^*  \d \bfx \\
&= \frac{1}{2} \int_{\mathbb{R}^8  \times \bbr_+^4}  \gamma_{\bfn}  \chi(|\bfx^* - \bfx| - (r + r^*)) ( \bfv(\bfx^*) - \bfv(\bfx))  \cdot \bfn(\bfx, \bfx^*)  F(\bfz) F(\bfz^*) \d\bfz^*  \d \bfz \d\bfx^*  \d \bfx \\
&= - \frac{1}{2} \int_{\mathbb{R}^8  \times \bbr_+^4}  \gamma_{\bfn}  \chi(|\bfx^* - \bfx| - (r + r^*)) ( \bfv(\bfx) - \bfv(\bfx^*))  \cdot \bfn(\bfx, \bfx^*)  F(\bfz) F(\bfz^*) \d\bfz^*  \d \bfz \d\bfx^*  \d \bfx \\
&=  - \frac{1}{2}\int_{\bbr^2  \times  \partial B(\bfx, 2r^{\infty})}  \gamma_{\bfn} \chi(|\bfx^* - \bfx| - (r + r^*))  ( \bfuu(\bfx) - \bfuu(\bfx^*))  \cdot \bfn(\bfx, \bfx^*)  \rho(\bfx) \rho(\bfx^*)  \d\sigma_{\bfx^*}  \d \bfx.
\end{aligned}
\end{align*}
(iii)~ This is a straightforward result from (i), (ii), and the definition \eqref{eq:he}.
\end{proof}

Similar to Theorem \ref{thm:E0}, we apply Barbalat's lemma and obtain the following assertion. 

\begin{theorem} 
Suppose that ocean velocity, initial data and system parameters satisfy
\[  \bfu = 0, \quad \E(0) < \infty,  \]
and let  $(\rho, \rho\bfuu, \rho E)$ be a global smooth solution to \eqref{eq:hpgmodel0}.
 Then, we have
\[ \lim_{t \to \infty} \E_K(t) = 0.  \]
\end{theorem}

\begin{proof} It follows from the above Lemma that 
\begin{align} \label{eq:dE}
\begin{aligned} 
\frac{d \E}{dt} &=  -  \int_{\bbr^2} \rho |\bfuu|^3 \d \bfx +  \frac{1}{2}\int_{\bbr^2  \times  \partial B(\bfx, 2r^{\infty})}   \gamma_{\bfv}  \big|  (\bfuu  - \bfuu^*) \cdot \bfn(\bfx, \bfx^*) \big |^2  \rho(\bfx) \rho(\bfx^*) \d\sigma_{\bfx^*} \d \bfx. \\
\end{aligned}
\end{align}  
We define energy production terms: 
\begin{align*}
\begin{aligned}
\Lambda(t) &:= \Lambda_1(t) + \Lambda_2(t),  \quad \Lambda_1 := \int_{\bbr^2} \rho |\bfuu|^3 \d \bfx,        \\
\Lambda_2 &:=   \frac{1}{2}\int_{\bbr^2  \times  \partial B(\bfx, 2r^{\infty})} |\gamma_{\bfv}|  \big|  (\bfuu  - \bfuu^*) \cdot \bfn(\bfx, \bfx^*) \big |^2  \rho(\bfx) \rho(\bfx^*) \d\sigma_{\bfx^*} \d \bfx.
\end{aligned}
\end{align*}
Then, we integrate the equation \eqref{eq:dE} to find 
\begin{equation*} 
\E(t) + \int_0^t \Lambda(s) ds = \E(0), \quad t \geq 0.
\end{equation*}
This yields
\begin{equation*} 
\sup_{0 \leq t < \infty} \E_K(t) \leq E(t) \leq \E(0), 
\end{equation*}
and 
\begin{equation*} 
\int_0^{\infty} \Lambda_1(t) dt = \int_0^{\infty} \int_{\bbr^2} \rho |\bfuu|^3 \d \bfx dt \leq \E(0) < \infty.
\end{equation*}
In order to use \rev{Barbalat's} lemma, it suffices to check that
\begin{equation*} 
\sup_{0< t < \infty} \Big| \frac{d}{dt} \Lambda_1(t)  \Big| < \infty. 
\end{equation*}
This implies the uniform continuity of the functional $\Lambda_1$.  
\rev{We use} \rev{Barbalat's} lemma \rev{to find} the desired estimate:
\begin{equation*}  
\lim_{t \to \infty} \E_K(t) = 0.
\end{equation*}
\end{proof}

\section{Numerical simulations} \label{sec:num}
In this section, we \rev{present} several numerical simulations for the particle model \rev{together with a case of the hydrodynamic model for comparison of the floe concentration in the regime of a large number of floes}. For numerical simulations, we use the forward Euler scheme \cite{butcher2016numerical} to discretize \eqref{eq:dem} and then verify the first two results of Lemma \ref{lem:sumparticle} and Theorem \ref{thm:v0} in the discrete setting.

\subsection{Forward Euler scheme for the particle model}
For $T \in (0, \infty)$, we divide the finite time interval $[0, T]$ into $N_t$ uniform time-step $\Delta t = T/N_t$. The discrete time points are denoted as $t_l = l\,\Delta t$ for $l = 0,1,\dots,N_t$. Let 
\[ X^i_l := \bfx^i(t_l) \quad \mbox{and} \quad V^i_l := \bfv^i(t_l) \]
denote the numerical solution of the position and velocity of particle $i$ at time $t_l$. Let $U^i_l := \bfu(t_l, \bfx^i)$ denote the ocean velocity in the position $\bfx^i$ at time $t_l$.  Starting from initial conditions $X^i_0 = \bfx^i(0)$ and $V^i_0 = \bfv^i(0)$, the forward Euler scheme updates the solution as follows:
\begin{align} 
\begin{aligned} \label{eq:dem_fe}
    X^i_{l+1} &= X^i_l + \Delta t\,V^i_l,  \quad l = 0,1,\dots, N_t-1,  \\
    V^i_{l+1} &= V^i_l + \frac{\Delta t}{m^i} \left[ \frac{1}{n} \sum_{j=1}^{n} \bff_c^{ij}(t_l) + \alpha^i (U^i_l - V^i_l) |U^i_l - V^i_l| \right].
\end{aligned}
\end{align}
The first equation $\eqref{eq:dem_fe}_1$ updates the position using the current velocity, while the second equation $\eqref{eq:dem_fe}_2$ updates the velocity using the force $\bfF^i(t_l)$. Note that the forward Euler method is an explicit time integration scheme, where the right-hand side of the ODE is evaluated at time $t_l$ and used to advance the solution to time $t_{l+1}$. The forward Euler scheme is explicit and first-order accurate in time. The local truncation error is $\mathcal{O}(\Delta t^2)$, and thus the global error is $\mathcal{O}(\Delta t)$, i.e.,
\begin{equation}  \label{eq:errfe}
|X^i_{N_t} - \bfx^i(T) | = \mathcal{O}(\Delta t), \qquad |V^i_{N_t} - \bfv^i(T) | = \mathcal{O}(\Delta t).
\end{equation}
We require a sufficiently small time-step $\Delta t$ to ensure accuracy and numerical stability, and we adopt this method here for simplicity and computational efficiency. 
Other numerical methods may be applied to solve the floe system.
The first-order accuracy of the numerical scheme ensures that the discretized simulation of our continuous-time particle model captures the essential qualitative features of sea ice floe dynamics described by \eqref{eq:dem}. While higher-order schemes may offer improved precision, the current level of accuracy is sufficient to reproduce key behaviors of the ice floe dynamics as well as to validate the theoretical findings established in Section \ref{sec:pmodel}

\subsection{Momentum and energy in a discrete setting}
In this subsection, we study the temporal evolution of total momentum and energy and compare them with analytical results in Lemma \ref{lem:sumparticle} and Theorem \ref{thm:v0}.
First, we define the discrete velocity moments as follows.
\begin{align}
\begin{aligned}  \label{eq:dMs}
& \M_0  = \sum_{i=1}^n m^i, \qquad \M_1 = \sum_{i=1}^n m^i V^i, \qquad \M_2 = \M_{2,V} + \M_{2,X}, \\
 & \M_{2,V} = \frac{1}{2} \sum_{i=1}^n m^i | V^i |^2, \qquad \M_{2,X} = \frac{1}{4n}\sum_{i,j=1}^n \kappa^{ij}_1  (\delta^{ij}(X^i, X^j))^2.
\end{aligned}
\end{align}
For the temporal evolution of total momentum, we use \eqref{eq:dem_fe} and \eqref{eq:dMs} to find 
\begin{align} 
\begin{aligned} \label{eq:dM1}
    \frac{\d\M_1}{\d t} \Big|_{t_l} & = \frac{\d}{\d t} \sum_{i=1}^n m^i V^i \Big|_{t_l}  = \sum_{i=1}^n m^i \frac{\d V^i}{\d t} \Big|_{t_l}  = \sum_{i=1}^n m^i \frac{ V^i_{l+1} - V^i_l }{\Delta t} \\
    & =  \frac{1}{n} \sum_{i, j=1}^{n} \bff_c^{ij}(t_l) + \sum_{i=1}^n \alpha^i (U^i_l - V^i_l) |U^i_l - V^i_l| \\
    & = \sum_{i=1}^n \alpha^i (U^i_l - V^i_l) |U^i_l - V^i_l|,
\end{aligned}
\end{align}
where the differential $d$ is understood as a discrete differential operator.  On the other hand, we use \eqref{eq:dem_fe} and \eqref{eq:dMs} to derive the temporal evolution of discrete energy functional:
\begin{align*}
\begin{aligned} 
    \frac{\d \M_{2,V}}{\d t} \Big|_{t_l}  & = \frac{1}{2} \sum_{i=1}^n m^i \frac{\| V^i_{l+1} \|^2 - \| V^i_l \|^2 }{\Delta t} \\
    & = \frac{1}{2} \sum_{i=1}^n m^i \frac{ V^i_{l+1} - V^i_l  }{\Delta t} \cdot (V^i_{l+1} + V^i_l  ) \\
    & = \frac{1}{2} \sum_{i=1}^n \Big( \frac{1}{n} \sum_{j=1}^{n} \bff_c^{ij}(t_l) + \alpha^i (U^i_l - V^i_l) |U^i_l - V^i_l| \Big) \cdot (V^i_{l+1} + V^i_l  ) \\
    & = \frac{1}{2} \sum_{i=1}^n \Big( \frac{1}{n} \sum_{j=1}^{n} \bff_c^{ij}(t_l) \Big) \cdot (V^i_{l+1} + V^i_l  ) + \frac{1}{2} \sum_{i=1}^n \Big( \alpha^i (U^i_l - V^i_l) |U^i_l - V^i_l| \Big) \cdot (V^i_{l+1} + V^i_l  ) \\
    & =  \frac{1}{n}  \sum_{i, j=1}^n  \big( \kappa^{ij}_1 \delta^{ij} + \kappa^{ij}_2 (V^i_l  - V^j_l )\cdot\bfn^{ij} \big) \bfn^{ij} \cdot  V^i_l \\
    & \quad + \frac{1}{2n}  \sum_{i, j=1}^n  \big( \kappa^{ij}_1 \delta^{ij} + \kappa^{ij}_2 (V^i_l  - V^j _l)\cdot\bfn^{ij} \big) \bfn^{ij} \cdot (V^i_{l+1}-V^i_l)  \\
    & \quad + \frac{1}{2} \sum_{i=1}^n \Big( \alpha^i (U^i_l - V^i_l) |U^i_l - V^i_l| \Big) \cdot (V^i_{l+1} + V^i_l  ) \\
    & = - \frac{\d \M_{2,X}}{\d t}  +  \frac{1}{2n} \sum_{i, j=1}^n \kappa^{ij}_2  \big( (V^i_l - V^j_l )\cdot\bfn^{ij} \big)^2   \\
    & \quad + \frac{1}{2n}  \sum_{i, j=1}^n  \big( \kappa^{ij}_1 \delta^{ij} + \kappa^{ij}_2 (V^i_l  - V^j _l)\cdot\bfn^{ij} \big) \bfn^{ij} \cdot (V^i_{l+1}-V^i_l)  \\
    & \quad + \frac{1}{2} \sum_{i=1}^n \Big( \alpha^i (U^i_l - V^i_l) |U^i_l - V^i_l| \Big) \cdot (V^i_{l+1} + V^i_l  ).
\end{aligned}
\end{align*}
This leads to
\begin{align} \label{eq:dM2}
\begin{aligned} 
    \frac{\d \M_2}{\d t}  & =   \frac{1}{2n} \sum_{i, j=1}^n \kappa^{ij}_2  \big( (V^i_l - V^j_l )\cdot\bfn^{ij} \big)^2   \\
    & \hspace{0.2cm} + \frac{1}{2n}  \sum_{i, j=1}^n  \big( \kappa^{ij}_1 \delta^{ij} + \kappa^{ij}_2 (V^i_l  - V^j _l)\cdot\bfn^{ij} \big) \bfn^{ij} \cdot (V^i_{l+1}-V^i_l)  \\
    &  \hspace{0.2cm} + \frac{1}{2} \sum_{i=1}^n \Big( \alpha^i (U^i_l - V^i_l) |U^i_l - V^i_l| \Big) \cdot (V^i_{l+1} + V^i_l  ).
\end{aligned}
\end{align}

\rev{We note that, due to the presence of the last two terms on the right-hand side of \eqref{eq:dM2}, the monotone decay of the total energy with a lower bound, as established in Lemma \ref{lem:sumparticle} for the time-continuous case, is not guaranteed in the discrete setting. Developing a numerical method that preserves this energy decay property with a lower bound remains an open problem.}

\subsection{Numerical validations}
In this subsection, we consider several simulation scenarios to validate the theoretical findings. 
We assume that the floe sizes and thickness do not change over time, and they follow the distributions \eqref{floe_size_pdf} and \eqref{floe_thickness_pdf}, respectively. 
For numerical simulations, we use $N=100$ floes and adopt a non-dimensionalized setting with a double periodic domain $[-\pi, \pi]^2$ in 2D. We set the model parameters (based on physical ones in \cite{herman2016discrete,li2020laboratory,herman2019wave}) $ E_e = 100, e_r = 0.15.$ For time discretization, we consider 
\[ T=20, \quad \Delta t = 10^{-3},  \quad N_t = 2\times 10^4. \]
The floe location and velocities are initialized randomly. We consider three simulation scenarios below. 

\begin{figure}[ht]
    \includegraphics[width=1.0\textwidth]{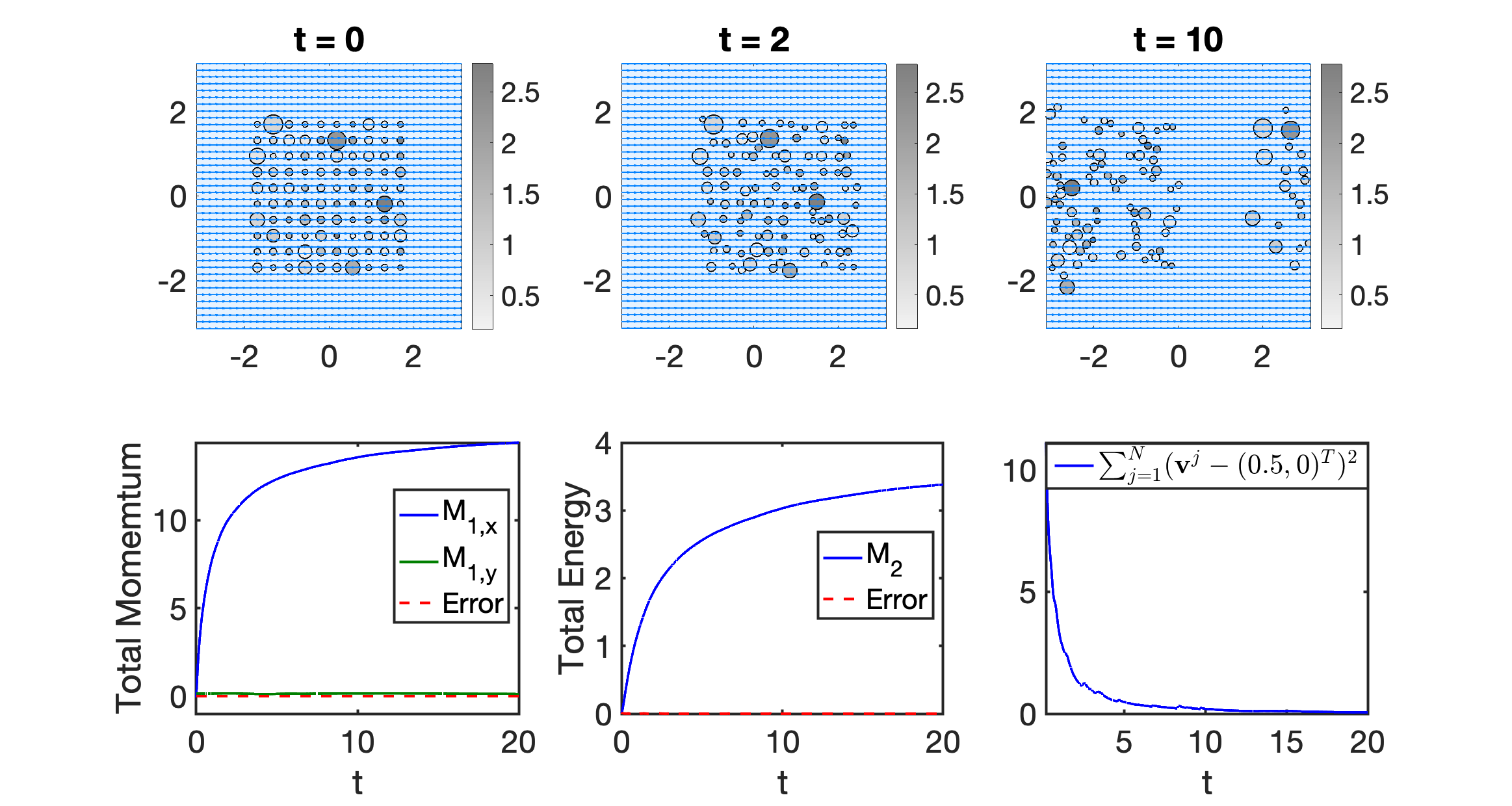}
    \caption{Total momentum and energy for the particle simulation of 100 floes when the ocean velocity is constant $\bfu=(0.5, 0)^T.$ }
    \label{fig:cv}
\end{figure}

In Figure \ref{fig:cv}, we show the simulation results when the ocean velocity is constant $\bfu=(0.5, 0)^T$ over space and time. In Figure \ref{fig:v1}, we consider the case when the ocean velocity is constant over time and varying over space, and in Figure \ref{fig:vt}, we consider the case when the ocean velocity varies over both space and time. 
\rev{For Figure \ref{fig:vt}, the ocean velocity is generated from a linear stochastic model based on Shallow water equations. 
This setting of a linear stochastic model for the Fourier mode of ocean current to approximate the nonlinear ocean dynamics captured by the Shallow water equation is a widely used and justified in, for example, \cite{majda2016introduction,farrell1993stochastic}. 
We use 80 Fourier modes for a complex ocean velocity field and refer to \cite{chen2021lagrangian, chen2022superfloe} for more details.}
We show the snapshots at times $t=0, 2, 10$ where the color bar shows the thickness of floes.

\begin{figure}[ht]
    \includegraphics[width=1.0\textwidth]{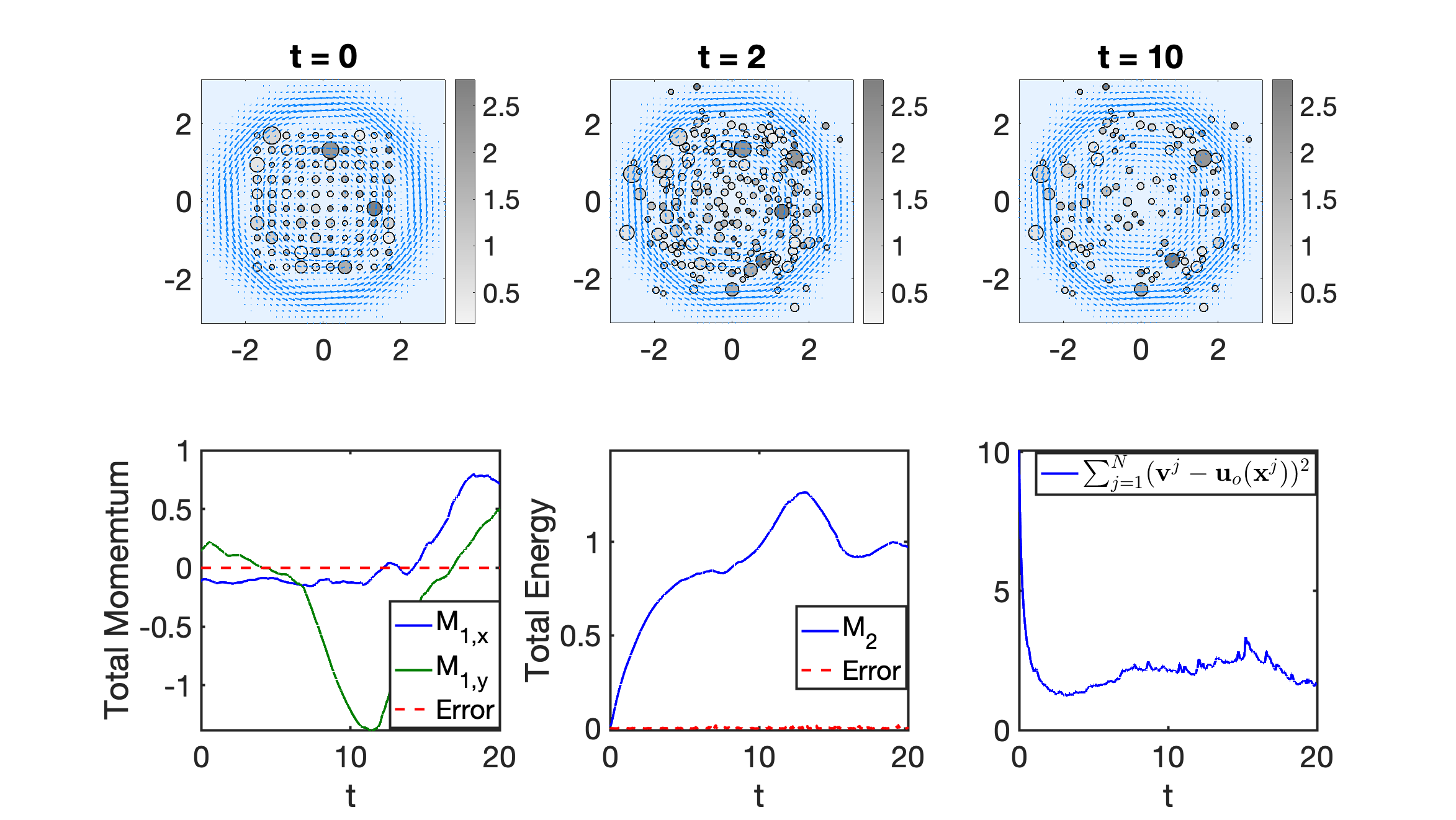}
    \caption{Total momentum and energy for the particle simulation of 100 floes when the ocean velocity is prescribed as $\bfu=(-ys, xs)^T$ with $s= (x^2+y^2-4)e^{(-(x^2+y^2)(x^2 + y^2 -8)/8)}/32.$ }
    \label{fig:v1}
\end{figure}

\begin{figure}[ht]
    \includegraphics[width=1.0\textwidth]{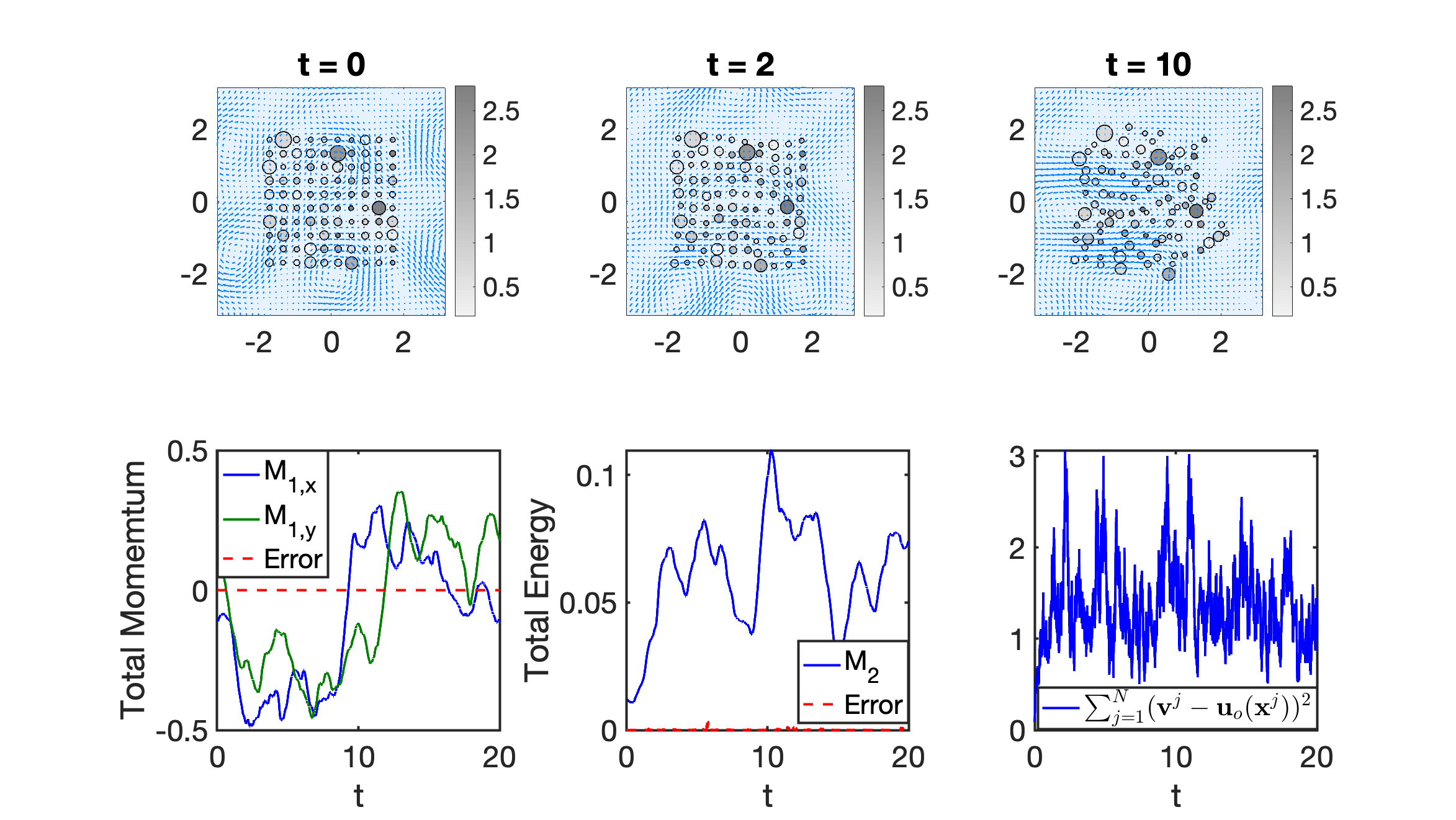}
    \caption{Total momentum and energy for the particle simulation of 100 floes when the ocean velocity is generated from a linear stochastic model based on \rev{shallow} water equations (c.f., \cite{chen2022superfloe}).}
    \label{fig:vt}
\end{figure}
In the bottom left plots of \rev{Figure \ref{fig:cv}, Figure \ref{fig:v1}, and Figure \ref{fig:vt}}, we show the total momentum in each component as well as the error defined as 
$$
\left|
\frac{\d\M_1}{\d t} \Big|_{t_l} -\sum_{i=1}^n \alpha^i (U^i_l - V^i_l) |U^i_l - V^i_l|
 \right|
$$
based on \eqref{eq:dM1}. These errors remain \rev{to be} zero over time; the small numerical errors are from the machine's precision and summation. This validates \eqref{eq:dM1} and the first result (i) of Lemma \ref{lem:sumparticle} in the discrete setting.

The bottom middle plots of \rev{Figure \ref{fig:cv} to Figure \ref{fig:vt}} show the total energy as well as the error defined similarly based on \eqref{eq:dM2} (simply, it is the absolution value error of the left-hand-size minus the right-hand side of \eqref{eq:dM2}).  Similarly, these errors remain zero over time; the small numerical errors are from the machine's precision and summation. This validates \eqref{eq:dM2} and the second result (ii) of Lemma \ref{lem:sumparticle} in the discrete setting. Herein, when considering the discrete discretization error for (ii) of Lemma \ref{lem:sumparticle}, an error estimate of order $\mathcal{O}(\Delta t)$ as in \eqref{eq:errfe} is observed. 

\rev{At} the bottom right plots of \rev{Figure \ref{fig:cv} to Figure \ref{fig:vt}}, we can see the closeness of the ice floes to the ocean currents in different scenarios: constant in both time and space, constant in time and varying in space, and varying in both time and space. Particularly, we note that the bottom right plots of Figure \ref{fig:cv} show the decaying of the velocities in terms of error $\sum_{j=1}^N |\bfv^j - \bfu(\bfx^j)|^2$ when the ocean is constant $\bfuu=(0.5, 0)^T$. This validates the Theorem \ref{thm:v0}: $\lim_{t\to \infty} \| \bfv^i - \bfu \|=0$. The bottom right plots of \rev{Figure \ref{fig:v1} and Figure \ref{fig:vt}} do not have this velocity converging property. This is expected as the ocean constant assumption in Theorem \ref{thm:v0} is not satisfied. 
 \rev{In Figure \ref{fig:vt}, the strong oscillations arise from the large variations in ocean velocity in this simulation case.}

\begin{figure}[ht]
    \includegraphics[width=1.0\textwidth]{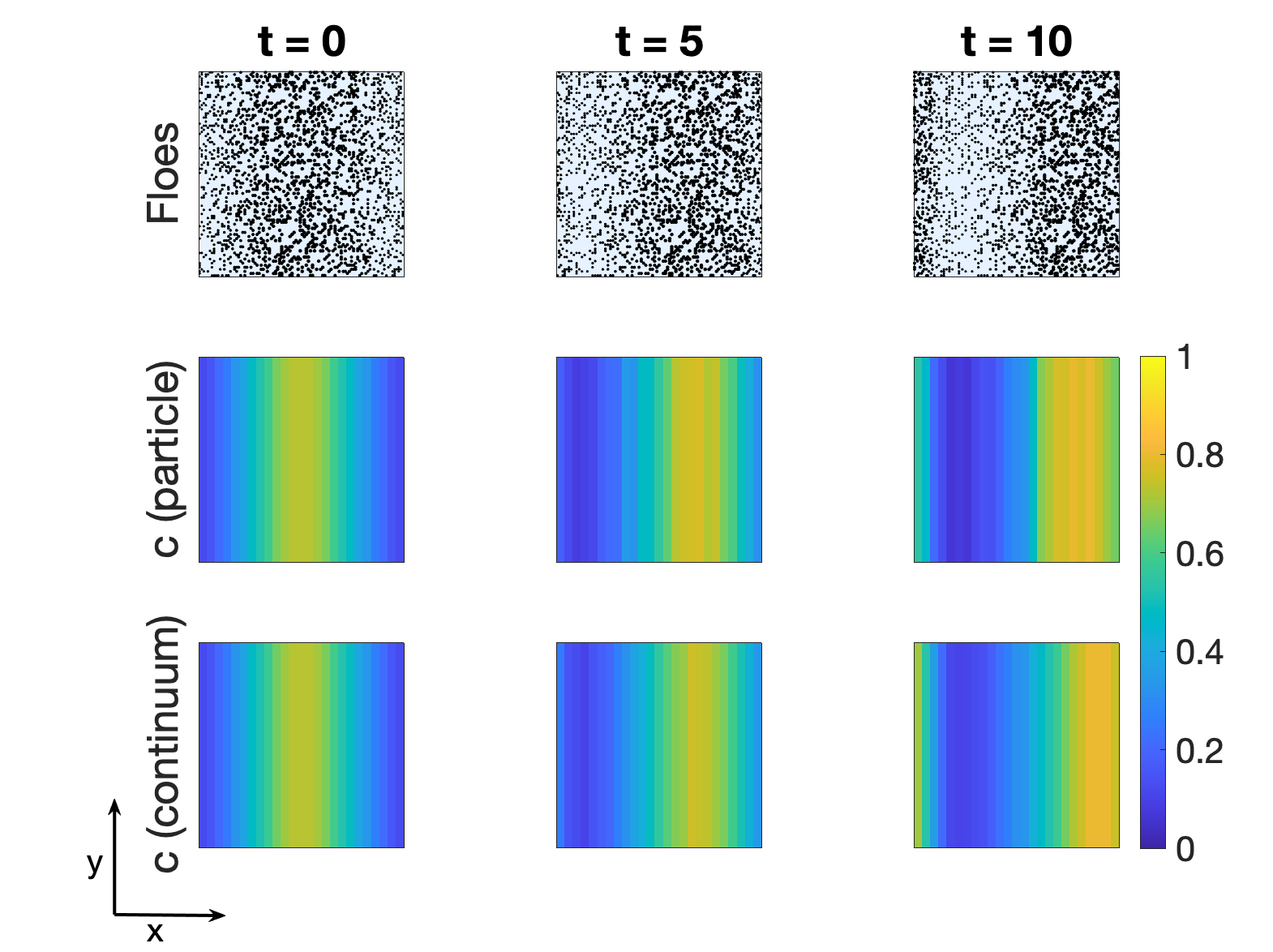}
    \vspace{-0.8cm}
    \caption{\rev{Comparison of concentration $c$ for particle and hydrodynamic models.   There are 10000 ice floes of different sizes, and 2000 floes were chosen randomly and plotted in the first row at time instances $t=0,5,10$. The second row shows the concentration of the floe particles in a uniform grid of size $25\times25$. The third row shows the concentration given by the hydrodynamic model solved on the same mesh. The ocean velocity is $\bfu=\big(0.1+0.1(x+\pi)(\pi-x), 0 \big)^T$. }}
    \label{fig:c}
\end{figure}

\rev{Lastly, Figure~\ref{fig:c} presents a comparison of the macroscopic behavior in concentration between the particle and hydrodynamic models in the regime of a large number of floes. In this experiment, we considered $10{,}000$ floes with varying sizes but uniform thickness, so that the mass density coincides with the floe concentration up to scaling factors. To mimic the gathering and scattering behavior studied in~\cite{deng2024particle}, the floes were advected by a spatially varying ocean velocity field 
$
\bfu = \big(0.1+0.1(x+\pi)(\pi-x), \, 0 \big)^T,
$
corresponding to a rightward flow with quadratic variation in $x$. For the hydrodynamic model, it is sufficient to solve the first two equations in (4.20), since the third equation is decoupled from the dynamics of energy $E$. 
For the numerical discretization, we employed the two-step Adams--Bashforth scheme with a fine time step size $\Delta t = 2\times 10^{-4}$ to ensure high accuracy in the temporal evolution. For the spatial discretization, we used a second-order central finite difference method on a uniform $25\times 25$ grid with double periodic boundary conditions. The first row of Figure~\ref{fig:c} shows the trajectories of $2000$ randomly chosen floes at time instances $t=0,5,10$ (plotting all $10{,}000$ floes would result in a completely dark figure). The second row reports the particle concentration, computed on each grid cell as the ratio of the total floe area contained in that cell to the cell area. We note that some non-smooth features appear in this representation, which are due to floes lying near the boundaries of grid cells, affecting the local evaluation of concentration. The third row displays the corresponding concentration obtained from the hydrodynamic model, solved on the same mesh. 
Overall, the hydrodynamic solution exhibits good agreement with the particle concentration at the macroscopic level, thereby validating the hydrodynamic approximation derived from the particle model in this setting. This agreement provides numerical evidence that the kinetic and hydrodynamic descriptions capture the essential large-scale dynamics of sea ice floes in regimes with a sufficiently large number of particles.}

\section{Concluding remarks} \label{sec:conclusion}
In this paper, we have introduced a theoretical framework for modeling sea ice floe dynamics, integrating particle, kinetic, and hydrodynamic approaches. Building on the Cucker-Smale particle models, we have derived the Vlasov-type kinetic description and a corresponding hydrodynamic formulation, enabling a systematic transition from discrete floe interactions to large-scale dynamics. In an idealized setting, we focused on non-rotating floes with linear velocity dynamics under oceanic drag forces, incorporating empirical laws for floe size and thickness distributions. Total momentum and energy are analyzed. Numerical simulations confirmed the \rev{potential} of this particle-continuum approach in capturing key sea ice behaviors in a multiscale setting. 

Future research directions include incorporating rotational dynamics, atmospheric drag,  thermodynamic effects such as \rev{melting/refreezing, floe fracturing, and more advanced contact laws in 3D. 
Another direction for future work is to investigate floe dynamics in regimes with turbulent ocean fields.
In numerical simulation, developing high-order numerical schemes may} further enhance accuracy and stability. 
\rev{It} remains unclear whether the forward Euler scheme preserves the energy dissipation property—that is, whether the right-hand side of \eqref{eq:dM2} remains negative in the absence of ocean drag forces. Developing numerical schemes that rigorously capture this dissipative behavior, ideally with a guaranteed positive lower bound, remains an open problem.
This is an interesting open question and will be left for future work. 
Additionally, validation against observational data will be crucial for assessing model reliability. Advancing these aspects will contribute to a deeper understanding of polar ice behavior with \rev{the potential for more accurate modelling and prediction} for climate science and oceanography.

\section*{Acknowledgements}
Q.D. is partially supported by the Australian National Computing Infrastructure (NCI) national facility under grant zv32 and the HMI Computing for Social Good Seed Research Grant., and the work of S.-Y. Ha is supported by National Research Foundation(NRF) grant funded by the Korea government(MIST) (RS-2025-00514472).

\bibliographystyle{siam}
\bibliography{ref}

\appendix{}
\section{A reduced kinetic model for floe dynamics} \label{app:a}
\rev{We} derive a kinetic model for floe dynamics where floe size and thickness distributions are provided as empirical probability density functions.  
In particular, according to observational data, the floe size distribution satisfies a power law \cite{stern2018seasonal}:
\begin{equation}\label{floe_size_pdf}
P(r) = a\frac{\kappa^a }{r^{a+1}}, \quad r > 0,
\end{equation}
where $r$ is the diameter of the floe, and $a$ and $\kappa$ are parameters. On the other hand, the floe thickness follows the Gamma distribution whose density function is given as follows (see \cite{bourke1987sea}):
\begin{equation}\label{floe_thickness_pdf}
\Gamma(h) = \frac{1}{\Gamma(k)\theta^k}h^{k-1}e^{-\frac{h}{\theta}}, \quad h > 0,
\end{equation}
where $k$ and $\theta$ are the shape and scale parameters, respectively. These are both common choices in practice. \rev{With this in mind,} we introduce a conditional probability density function 
\begin{equation}  \label{eq:condp}
F(t, \bfx,\bfv, r, h) = \rev{f(t, \bfx,\bfv | r, h)} P(r) \Gamma(h). 
\end{equation}
We substitute the above ansatz \eqref{eq:condp} into \eqref{eq:p-23} to derive a kinetic equation for conditional distribution function $f$:
\begin{equation} \label{eq:p-35}
 \partial_t f +   \bfv  \cdot \nabla_{\bfx} f  +   \nabla_{\bfv} \cdot (\bff[f] f )= 0,
\end{equation}
where 
\begin{align}
\begin{aligned}  \label{eq:p-36}
& \bff[f]:=  \bff_o +  \bff_{c}[f] , \bff_{c}[f]  = \bff_{c,\bfv}[f] +  \bff_{c,\bfn}[f], \quad  \bff_o(\bfv) :=  \gamma_o \left(\bfu -\bfv \right)\left| \bfu -\bfv \right|, \\
& \bff_{c,\bfn}[f](t,\bfz) :=  \int_{\bbr^{4} \times \bbr_+^{2}}  \gamma_{\bfn}  \chi(|\bfx^* - \bfx| - (r + r^*)) \bfn(\bfx, \bfx^*)  f(t, \bfz^*)  P(r^*) \Gamma(h^*) d\bfz^*, \\
& \bff_{c,\bfv}[f](t,\bfz) :=  \int_{\bbr^{4} \times \bbr_+^{2}}   \gamma_{\bfv} [ (\bfv  - \bfv^*) \cdot \bfn(\bfx, \bfx^*)] \bfn(\bfx, \bfx^*) f(t, \bfz^*)  P(r^*) \Gamma(h^*) d\bfz^*.
\end{aligned}
\end{align}
Following the previous subsections, similar momentum and energy estimates of the kinetic model \eqref{eq:p-35}-\eqref{eq:p-36} can be derived for this setting.

\end{document}